\RequirePackage{fix-cm,amsmath}
\documentclass[smallextended]{svjour3}       
\usepackage{graphicx}
\usepackage[utf8]{inputenc}
\usepackage[english]{babel}
\usepackage{wrapfig}
\usepackage[T1]{fontenc}
\usepackage{mathptmx}
\usepackage[english]{babel}
\usepackage{amsmath}
\usepackage{mathtools}
\usepackage{amsfonts}
\usepackage{amssymb}
\usepackage{makeidx}
\usepackage{scalerel}
\usepackage{hyperref}
\hypersetup{
    colorlinks=true,
    linkcolor=black,
    urlcolor=blue,
    citecolor=black
}
\usepackage[round]{natbib}
\usepackage[left]{lineno}
\usepackage{comment}
\usepackage{placeins}
\usepackage{xfrac}
\usepackage{color}
\usepackage{float}
\usepackage{caption}
\usepackage{subcaption}
\usepackage{chngcntr}
\usepackage{setspace}
\usepackage{array}
\usepackage{lscape}
\usepackage[labelsep=space]{caption}
 \setcitestyle{aysep={}} 
\counterwithin{figure}{section}
\counterwithin{table}{section}

\usepackage{todonotes}

\newcommand{\orcid}[1]{\href{https://orcid.org/#1}{\includegraphics[width=8pt]{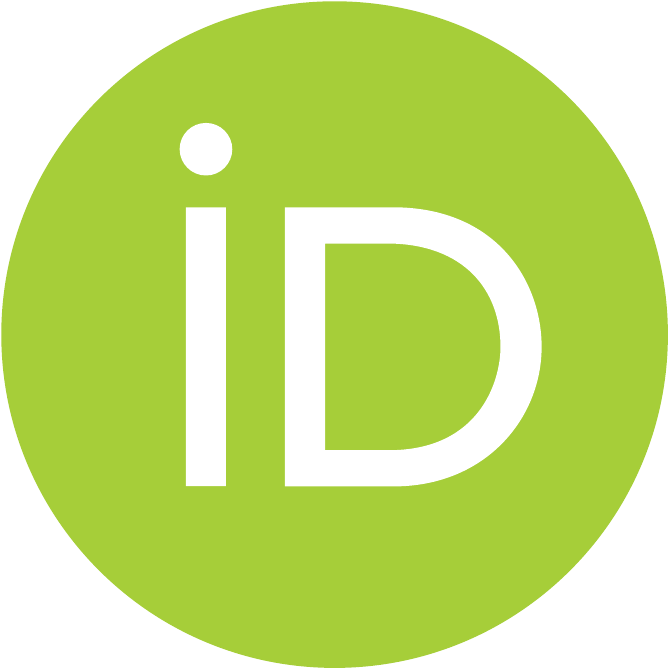}}}
\begin{document}
\title{ Dynamical Model of Mild Atherosclerosis: Applied Mathematical Aspects
}


\author{Debasmita Mukherjee\orcid{0000-0003-1735-0421}         \and
       Sishu Shankar Muni\orcid{0000-0001-9545-8345}   \and Hammed Olawale Fatoyinbo \orcid{0000-0002-6036-2957}  
}



\institute{ D. Mukherjee \at
              School of Mathematics, Applied Statistics \& Analytics\\
SVKM'S NMIMS (Deemed to be) University
Mumbai 400056, India \\
              \email{debasmita.sarada@gmail.com}           
           \and
           S. S. Muni \at
            1.~ Department of Physical Sciences,
Indian Institute of Science and Educational Research Kolkata,
Campus Road, Mohanpur, West Bengal, 741246,
India\\
2.~ School of Mathematical and Computational Sciences,
Massey University,
Colombo Road, Palmerston North, 4410,
New Zealand\\
   \email{sishu1729@iiserkol.ac.in, ssmuni760010@gmail.com} 
              \and
               H. O. Fatoyinbo \at
               1.~ EpiCentre, School of Veterinary Science, Massey University, Palmerston North 4442, New Zealand
             \\2.~ School of Mathematical and Computational Sciences, Massey University, Colombo Road, Palmerston North, 4410,
New Zealand \\ \email{h.fatoyinbo@massey.ac.nz} 
}

\date{Received: date / Accepted: date}

\maketitle

\begin{abstract}
Atherosclerosis is a chronic inflammatory disease occurs due to plaque accumulation in the inner artery wall. In atherosclerotic plaque formation monocytes and macrophages play a significant role in controlling the disease dynamics. In the present article, the entire biochemical process of atherosclerotic plaque formation is presented in terms of an autonomous system of nonlinear ordinary differential equations involving concentrations of oxidized low-density lipoprotein (LDL), monocytes, macrophages, and foam cells as the key dependent variables. To observe the capacity of monocytes and macrophages the model has been reduced to a two-dimensional temporal model using quasi steady state approximation theory. Linear stability analysis of the two-dimensional ordinary differential equations (ODEs) model has revealed the stability of the equilibrium points in the system. We have considered both one- and two- parameter bifurcation analysis with respect to parameters associated to the rate at which macrophages phagocytose oxidised LDL and the rate at which LDL enters into the intima. The bifurcation diagrams reveal the oscillating nature of the curves representing concentration of monocytes and macrophages with respect to significant model parameters. We are able to find the threshold values at which the plaque accumulation accelerates in an uncontrollable way. Further to observe the impact of diffusion, a spatiotemporal model has been developed. Numerical investigation of the partial differential equations (PDEs) model reveals the existence of travelling wave in the system which ensures the fact that the development of atherosclerotic plaque formation follows reaction-diffusion wave.
\end{abstract}

\section{Introduction}
Atherosclerosis is a procedure linked with inflammation. It is characterized by the fat-laden foam cells and cellular debris. It is a cardiovascular disease resulting from hypertension, smoking, diabetes, high level of cholesterol, and age. The record suggests that every year nearly 900,000 deaths occur due to atherosclerosis in western countries \citep{hao2014ldl}.

The atherosclerotic plaques can occur in artery walls throughout the body. It usually emerges in the artery branches where there is low endothelial shear stress (ESS), such as at junctions and curves \citep{chatzizisis2007role}. Endothelial cells (ECs) tend to change their morphology under the influence of wall shear stress (WSS). Under high wall shear stress, the shape of the ECs is elongated. Whereas, under low wall shear stress, the ECs shape turns to be polygon or round without any definite alignment pattern. The morphological changes of ECs along with the low ESS may be the reason for the expansion of the junctions between ECs. These little "gaps" between ECs, with flow stagnation and subsequent extension of the residence time of circulating low-density lipoprotein (LDL) facilitate sub-endothelial infiltration of LDL. Inside the vessel walls, LDL particles get oxidized by free radicals, which are released by the damaged endothelium. The oxidized LDL stimulates ECs to express adhesion molecules, for example vascular cell adhesion molecule I (VCAM-I). The chemoattractant and VCAM-I attracts monocytes, and T-cells from the lumen \citep{chatzizisis2007role}. 
 \begin{figure}[htbp]
   \centering
   \includegraphics[scale=0.4]{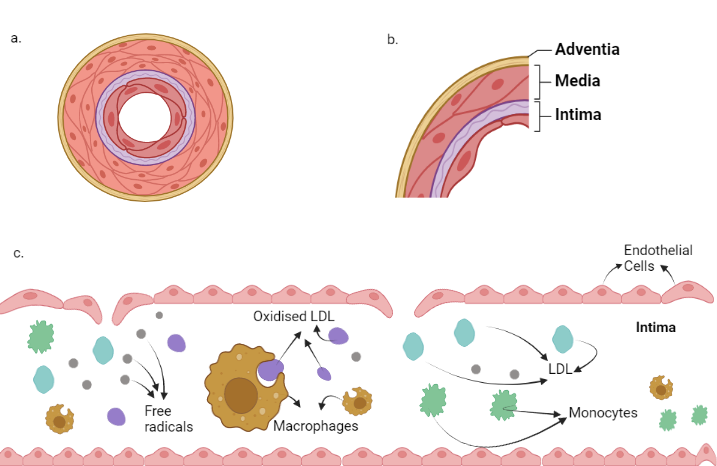}
   \caption{a. Cross section of an artery. b. Artery layers. c. Schematic diagram of atherosclerotic plaque formation}
    \label{fig:intro1}
  \end{figure}
Inside the intima, monocytes convert into macrophages in the presence of scavenger receptors. The macrophages digest (or phagocytose) oxidized LDL particles, and together they form the foam cells. The cycle of monocyte recruitment to foam cell formation continues, and it gradually increases the volume of the plaque. Once the plaque is formed, a fibrous cap is generated surrounding the plaque. The plaque volume and the WSS of the arterial wall are proportional to each other. A gradual increase in the plaque volume escalates the WSS of the arterial wall. Eventually, WSS reaches its threshold and breaks down. Then the atherosclerotic plaque bulges into the lumen and causes hindrances in the smooth blood flow. It leads to thrombosis in that inflammated region, and in the long term it may result in a serious myocardial infection \citep{libby2002inflammation}. This entire biochemical process is presented in Fig.~\ref{fig:intro1}. Fig.~\ref{fig:intro1} is created in \href{https://biorender.com/}{biorender}.

Atherosclerosis involves several cellular interactions of which macrophage phagocytosis plays a significant role in the plaque, formation process \citep{mukherjee2019reaction}. Macrophages absorb and remove pathogenic material inside the body through a process known as phagocytosis. In atherosclerotic plaque formation macrophages ingest and remove oxidized LDL. Macrophages also absorb apoptotic material and necrotic material, including cellular debris and free lipids \citep{moore2013macrophages}. Cells that undergo apoptosis, sends out chemokines to attract macrophages (for "find- me" signals) and then produce "eat - me" receptors on their cell membrane. This accelerates the ingestion process by macrophages. Similarly, macrophages absorb oxidized LDL via a scavenger receptor. It has been observed that necrotic materials are toxic to macrophages and hence do not facilitate the phagocytosis process \citep{kojima2017role}.

Over the past few years, several clinical investigations in terms of mathematical and statistical models are analysed to gain a clear insight of the disease dynamics. Till date the complex biological process behind atherosclerosis is partly understood. Analysis of some recent mathematical models can be found in \citep{zohdi2004phenomenological,el2007atherosclerosis,fok2012mathematical,bulelzai2012long,hao2014ldl,chalmers2015bifurcation,friedman2015mathematical,anlamlert2017modeling, mukherjee2020dynamical,mukherjee2022dynamical}.

\cite{bulelzai2012long} described two-dimensional ODE model of atherosclerosis involving oxidized LDL, monocytes, macrophages, and foam cells. They have done a comparative study between two models treating WSS as a parameter in one model and as a time-dependent function in the other model. They have shown that the drastic decrease in lumen radius is quite sensitive to the value of WSS. \cite{bulelzai2014bifurcation} further investigated the progression of atherosclerosis via numerical bifurcation analysis of the two models described in \citep{bulelzai2012long}.

The purpose of this paper is to give in-depth investigation of monocyte and macrophage proliferation in atherosclerotic plaque forming process. We begin in Sect.~\ref{s1} with the model of \cite{bulelzai2014bifurcation} for atherosclerotic plaque evolution. In the current investigation oxidised LDL has been assumed to be in steady state condition. We apply quasi-steady state approximation (QSSA) to reduce the model to a two-dimensional ODE system in Sect.~\ref{ss1}. Usually the development of atherosclerosis follows a reaction-diffusion process \citep{el2012reaction, chalmers2015bifurcation}. To capture this natural phenomenon the reduced model is extended to a reaction-diffusion systems in Sect.~\ref{ss2}.

Then in Sect.~\ref{s2} we analyse the dynamics of the reduced ODE model. We find the properties of the equilibrium points of the systems via linear stability analysis in \ref{ss22}. A detailed numerical bifurcation analysis is carried out in Sect.~\ref{sec:bifurcation}. In Sect.~\ref{s3} the spatiotemporal dynamics is discussed and the physical interpretation of results is provided in Sect.~\ref{s4}. Finally, conclusions are presented in Sect.~\ref{s5}. 



\section{Model formulation} \label{s1}
Macrophage phagocytosis is a fundamental function in atherosclerotic plaque formation. The main objective of this paper is to investigate the role of monocytes and macrophages in phagocytosis and hence in evolution of atherosclerotic plaques. At first, we have  considered the same model as in \citep{bulelzai2014bifurcation}
\begin{eqnarray}
     \dot{m} &=& (\frac{aL}{(1+\sigma)(1+L)} - \epsilon - c)m,\label{eq:1}\\
     \dot{M} &=& cm - \frac{bML}{1+L},\label{eq:2}\\
     \dot{L} &=& \frac{dm}{f+m} - eLM -L, \label{eq:3}\\
     \dot{F} &=& \frac{bLM}{1+L} \label{eq:4},
 \end{eqnarray}
 where $m$, $M$, $L$, and $F$ denote the concentration of monocytes, macrophages, oxidized LDL, and foam cells respectively. All parameters of this model are dimensionless and non-negative. Physical interpretation of this system \eqref{eq:1}-\eqref{eq:4} is the same as in \citep{bulelzai2014bifurcation}. As equation \eqref{eq:4} decouples the above system so in \citep{bulelzai2014bifurcation}, the authors have only considered the evolution equations of $m, M $, and $L$. We have also considered the model constituting from the equations \eqref{eq:1}-\eqref{eq:3} for our investigation.
 
 \subsection{\textbf{Model reduction}}\label{ss1}
 
 Quasi-steady state approximation (QSSA) theory is a mathematical method for streamlining Differential Equations. It is one of the poignant appliances available for studying several complex biological models. For example, Zhang et al. (\citeyear{zhang2013foam}) have used QSSA  to reduce ODE models of atherosclerosis. We have also applied QSSA to investigate the nature of monocytes and macrophages at the steady state value of LDL (L).
 
  The time series plot in Fig.~\ref{fig4} depicts that the concentrations of monocytes and macrophages are evolving at a larger scale in comparison with oxidized LDL. Thus the model in \eqref{eq:1}-\eqref{eq:3} can be reduced into a two-species model involving monocytes and macrophages as the dependent variable while oxidized LDL is assumed to be in steady state.
 \begin{table}
\centering
\caption{List of parameter values used in the model}
\scalebox{0.7}{%
\begin{tabular}{|c|c|c|c|}
\hline
Parameters & Description & Numeric values & Source\\
\hline
$a$ & rate of monocyte insertion inside the intima & 1& \cite{bulelzai2014bifurcation}\\ 
 \hline
 $b$ & rate at which macrophage phagocytose oxidized LDL & 0.2 & present study\\ 
 \hline
 $c$ & rate at which monocytes differentiate into macrophages & 0.05 & \cite{bulelzai2014bifurcation}\\ 
 \hline
 $d$ & rate of total flux of LDL cholesterol entering into
the plaque region & 3 & present study\\ 
 \hline
 $e$ & rate of ingestion of oxidized LDL by macrophages & 1 & \cite{bulelzai2014bifurcation}\\ 
 \hline
 $f$ & saturation constant & 1 & \cite{bulelzai2014bifurcation}\\ 
 \hline
 $\epsilon$ & death rate of monocytes & 0.01 & \cite{bulelzai2014bifurcation}\\ 
 \hline
 $\sigma$ & wall shear stress & 1 & \cite{bulelzai2014bifurcation}\\ 
 \hline
\end{tabular}
}
\label{table:t1}
\end{table}

  \begin{figure*}[htbp]
   \hspace*{-2cm}
\centering
  \begin{subfigure}[b]{.5\linewidth}
    \centering
    \caption{}
   \includegraphics[width=\textwidth]{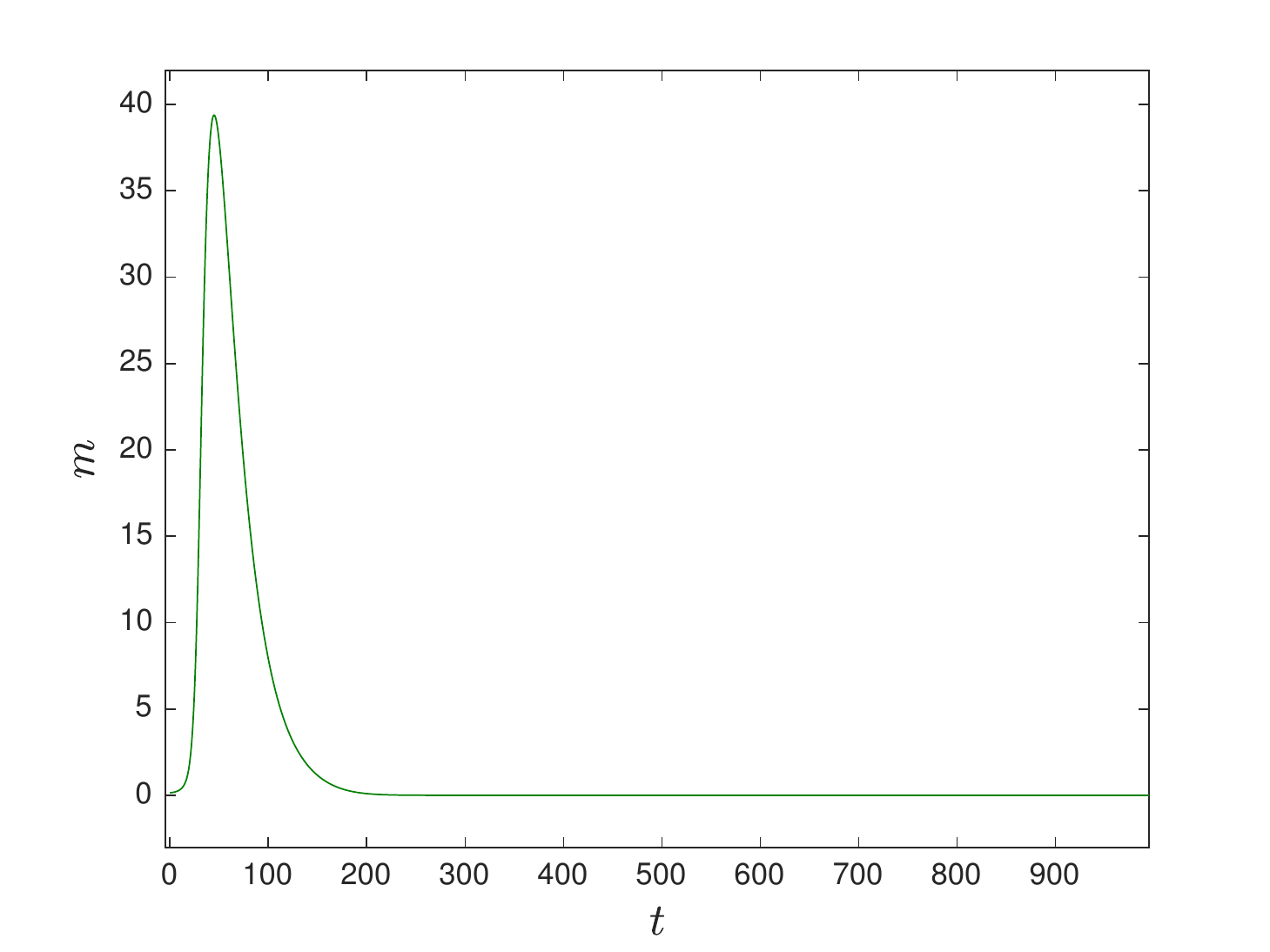}
    \label{fig:m0251}
  \end{subfigure}%
  \begin{subfigure}[b]{.5\linewidth}
    \centering
   \caption{}
   \includegraphics[width=\textwidth]{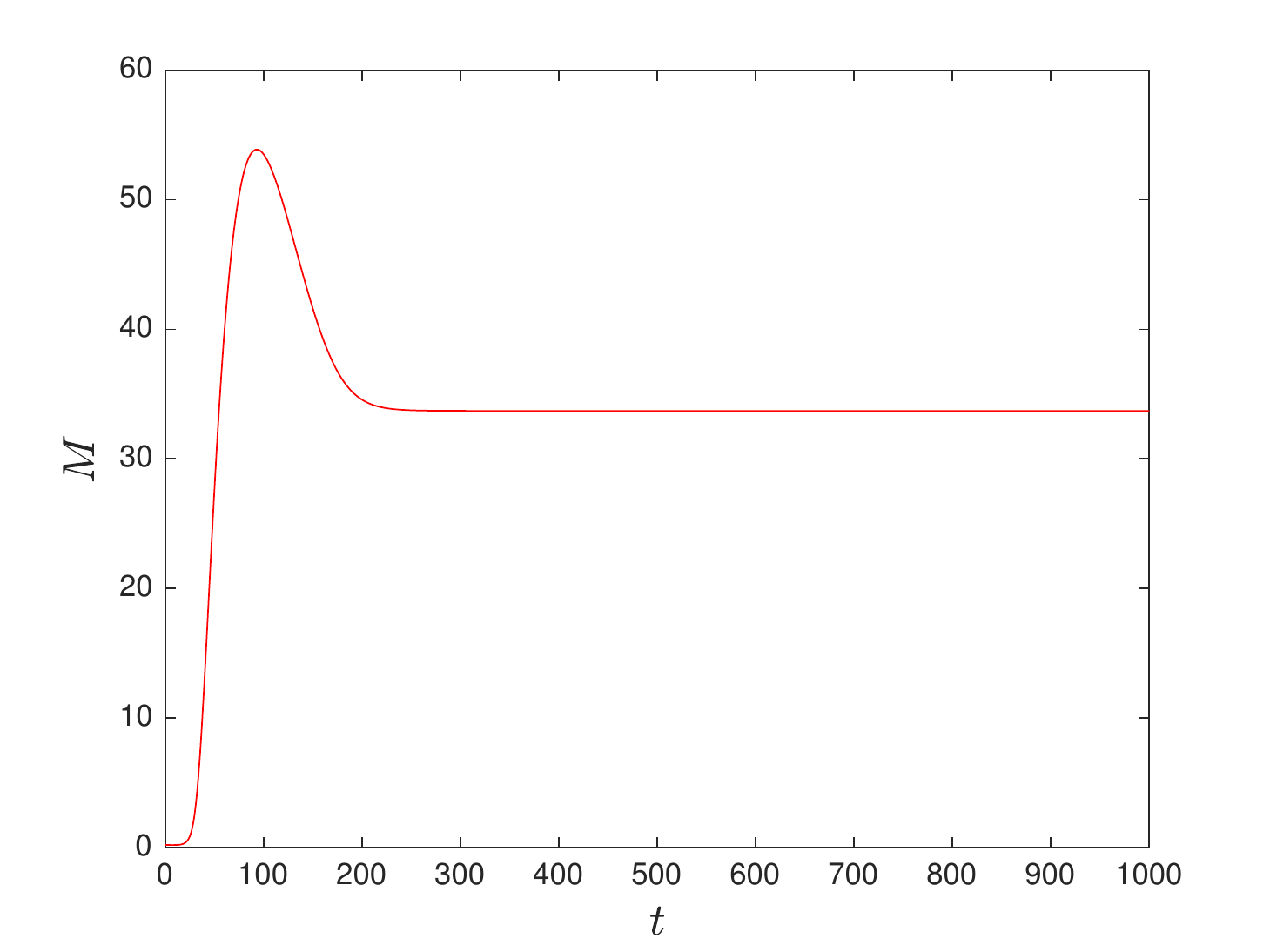}
   \label{fig:M0251}
  \end{subfigure}\\%
  \begin{subfigure}[b]{.5\linewidth}
    \centering
   \caption{}
   \includegraphics[width=\textwidth]{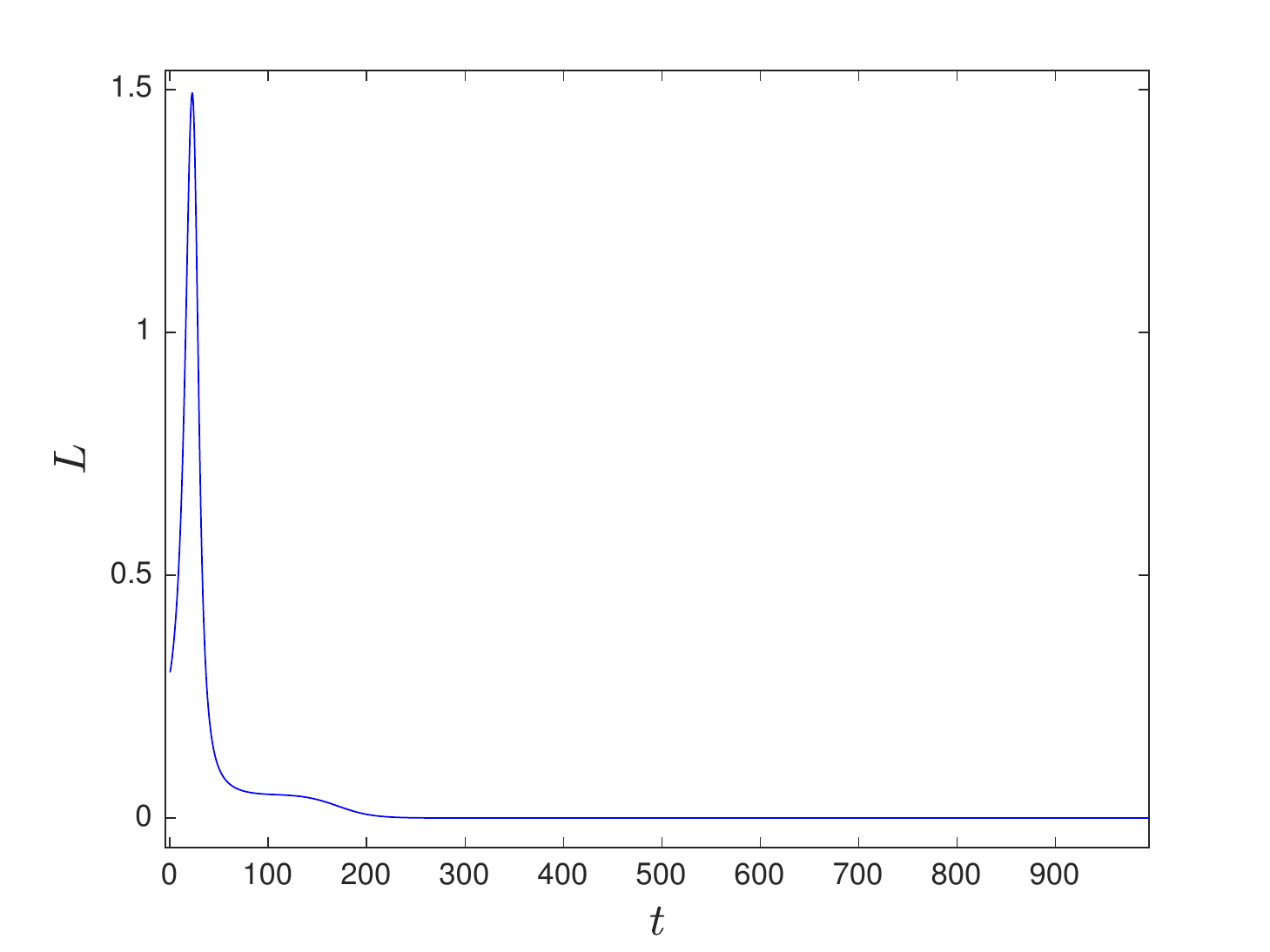}
   \label{fig:PP0251}
  \end{subfigure}%
  \caption{Time series plot of the state variables $m,M,L$. The plot shows that the state variable $L$ is a slow variable. The parameters of the system \eqref{eq:1}-\eqref{eq:4} are set as $d = 3, a = 1, \sigma = 1, \epsilon = 0.01, c = 0.05, f = 1, e = 1, b = 0.2$.\textcolor{black}{The parameter values of the system  \eqref{eq:1}-\eqref{eq:4} are provided in Table \ref{table:t1}} Initial conditions set as $m= 0.14, M= 0.2, L= 0.3$}
  \label{fig4}
  \end{figure*}
  
The quasi steady state value of L is found to be
\begin{equation*}
    L_s= {\frac {dm}{ \left( f+m \right)  \left( Me+1 \right) }},
\end{equation*}
replacing $L$ by $L_s$ in model equation \eqref{eq:1} and \eqref{eq:2}, we obtain:
\begin{eqnarray}
\frac{dm}{dt}&=& \left( {\frac {adm}{ \left( f+m \right)  \left( Me+1 \right)  \left( 
1+\sigma \right) } \left( 1+{\frac {dm}{ \left( f+m \right)  \left( Me
+1 \right) }} \right) ^{-1}}-\epsilon-c \right) m, \label{eq:5}\\
\frac{dM}{dt}&=& cm-{\frac {bMdm}{ \left( f+m \right)  \left( Me+1 \right) } \left( 1+{
\frac {dm}{ \left( f+m \right)  \left( Me+1 \right) }} \right) ^{-1}}. \label{eq:6}
\end{eqnarray}

\subsection{\textbf{PDE formulation of the model}}\label{ss2}
In order to capture the diffusivity of monocytes and macrophages, we have considered a one-dimensional cartesian domain representing a cross-section of the intima of the affected artery away from the edges of the plaque. This domain is bounded by the endothelium at $x=0$ and by the Internal elastic lamina $x=L$. So, the dependent variables $m$ and $M$ are functions of position $x \in [0,L]$ and time $t \geq 0$. We denote the concentration gradient of monocytes and macrophages by $m(x,t)$, and $M(x,t)$ respectively.

Following is a simplified reaction-diffusion model on an interval $ \Omega =[0,L] \subseteq \mathbb{R}$, where we have assumed that monocytes and macrophages are diffusing according to Fick's law in $\Omega$,

\begin{eqnarray}
\frac{\partial m}{\partial t}&=&D_{1}\frac{\partial^2 m}{\partial x^2}+ m\left( -\epsilon-c\right)\notag\\
&&+{\frac {adm^2}{ \left( f+m \right)  \left( Me+1 \right)  \left( 
1+\sigma \right) } \left( 1+{\frac {dm}{ \left( f+m \right)  \left( Me
+1 \right) }} \right) ^{-1}} ,\label{eq:m}\\
\frac{\partial M}{\partial t}&=&D_{2}\frac{\partial^2 M}{\partial x^2}+cm-{\frac {bMdm}{ \left( f+m \right)  \left( Me+1 \right) } \left( 1+{
\frac {dm}{ \left( f+m \right)  \left( Me+1 \right) }} \right) ^{-1}}, \label{eq:M}
\end{eqnarray}

 where $D_{1}$ and $D_{2}$ are diffusion constants. The initial conditions are 
\begin{align}
    m(x,0)=m_{0}(x)~\text{and}~M(x,0)=M_{0}(x)~\forall x\in\Omega ,
\end{align}
for some given functions $m_{0}(x)$ and $M_{0}(x)$. To visualize some self-organised patterns zero flux boundary conditions 

$$\frac{\partial m}{\partial n}= \frac{\partial M}{\partial n} =0,$$ in $\partial \Omega \times (0, \infty )$, where $n$ is the outward normal vector of the boundary $\partial \Omega$, which is assumed to be smooth, is assumed. Biologically it means that monocytes and macrophages do not leave the domain during the time of investigation and hence pointing towards the choice of zero-flux boundary condition. All the parameters are chosen to be dimensionless in the above \eqref{eq:m}-\eqref{eq:M} model.
 
\section{Dynamics of the reduced ODE model} \label{s2}
To investigate how the dynamics of the system \eqref{eq:5}-\eqref{eq:6} evolves with time via divergence we have performed the following analysis. If the divergence is zero, it is referred to as a conservative system. If the divergence is negative, the system is said to be dissipative and in this case, it can support the existence of attractors. For a positive divergence, the system increases in volume in phase space and might diverge.

\textcolor{black}{For system \eqref{eq:5}-\eqref{eq:6}, the divergence is given by
\begin{equation}
 \frac{\partial{\dot{m}}}{\partial{m}} + \frac{\partial{\dot{M}}}{\partial{M}} = -\epsilon - c + \frac{adf(1+\sigma)(Me+1)}{(1+\sigma)^2 ((f+m)(Me+1)+dm)^2} - \frac{bd((f+m)m + (f+m)dm^2)}{((f+m)(Me+1) + dm)^2}.
 \label{eq:divergence}
\end{equation}
Observe that the divergence of the system depends on the parameters and state variables of the system. To evaluate the divergence, we can compute a set of local divergence along the trajectory of the phase space for a particular set of parameter values and obtain the average over time. For the parameters chosen as $a=1, b=0.318, c=0.05, d = 0.9, e = 1,f = 1, \epsilon = 0.01, \sigma = 1$, we observe that the divergence of the system is $ -0.6886<0$, see Fig.\ref{fig:1pard=3}. From this, we conclude that the system is dissipative and can support attractors. } It represents that the system has multiple steady states, and hence it depicts bistability.

 \begin{figure}[htbp]
   \centering
    \includegraphics[width=0.6\textwidth]{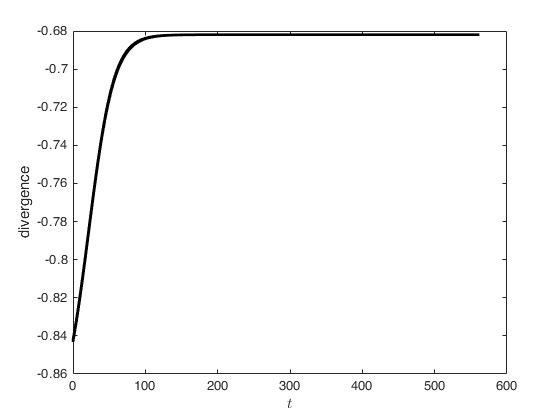}
    \caption{Evolution of the local divergence with time for system \eqref{eq:5}-\eqref{eq:6}. The parameters are chosen as $a=1, b=0.318, c=0.05, d = 0.9, e = 1,f = 1, \epsilon = 0.01, \sigma = 1$ }
    \label{fig:divtime}
  \end{figure}
  
 \begin{figure*}[htbp]
\centering
  \begin{subfigure}[b]{.5\linewidth}
    \centering
    \caption{}
   \includegraphics[width=\textwidth]{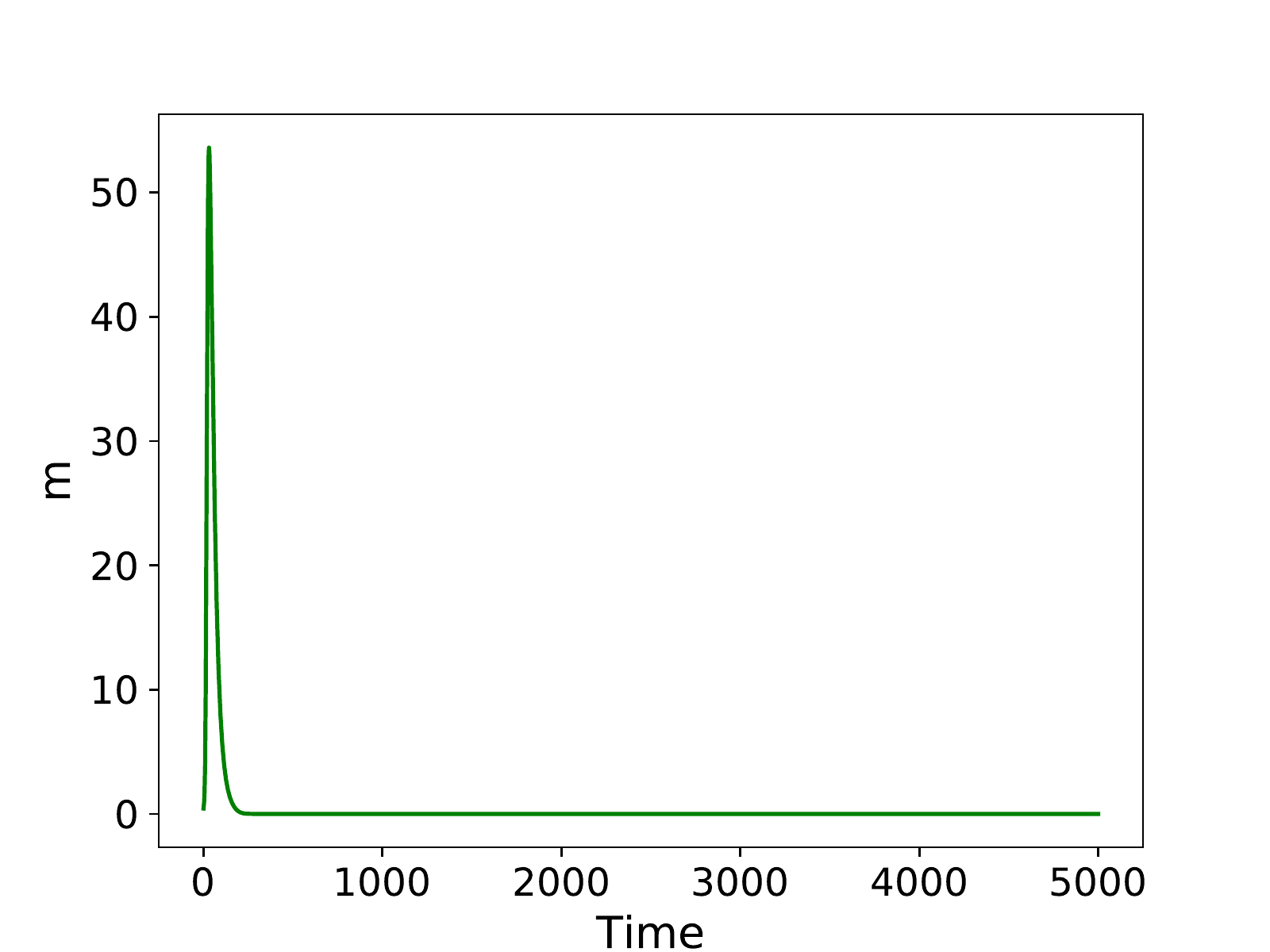}
    \label{fig:m025}
  \end{subfigure}%
  \begin{subfigure}[b]{.5\linewidth}
    \centering
   \caption{}
   \includegraphics[width=\textwidth]{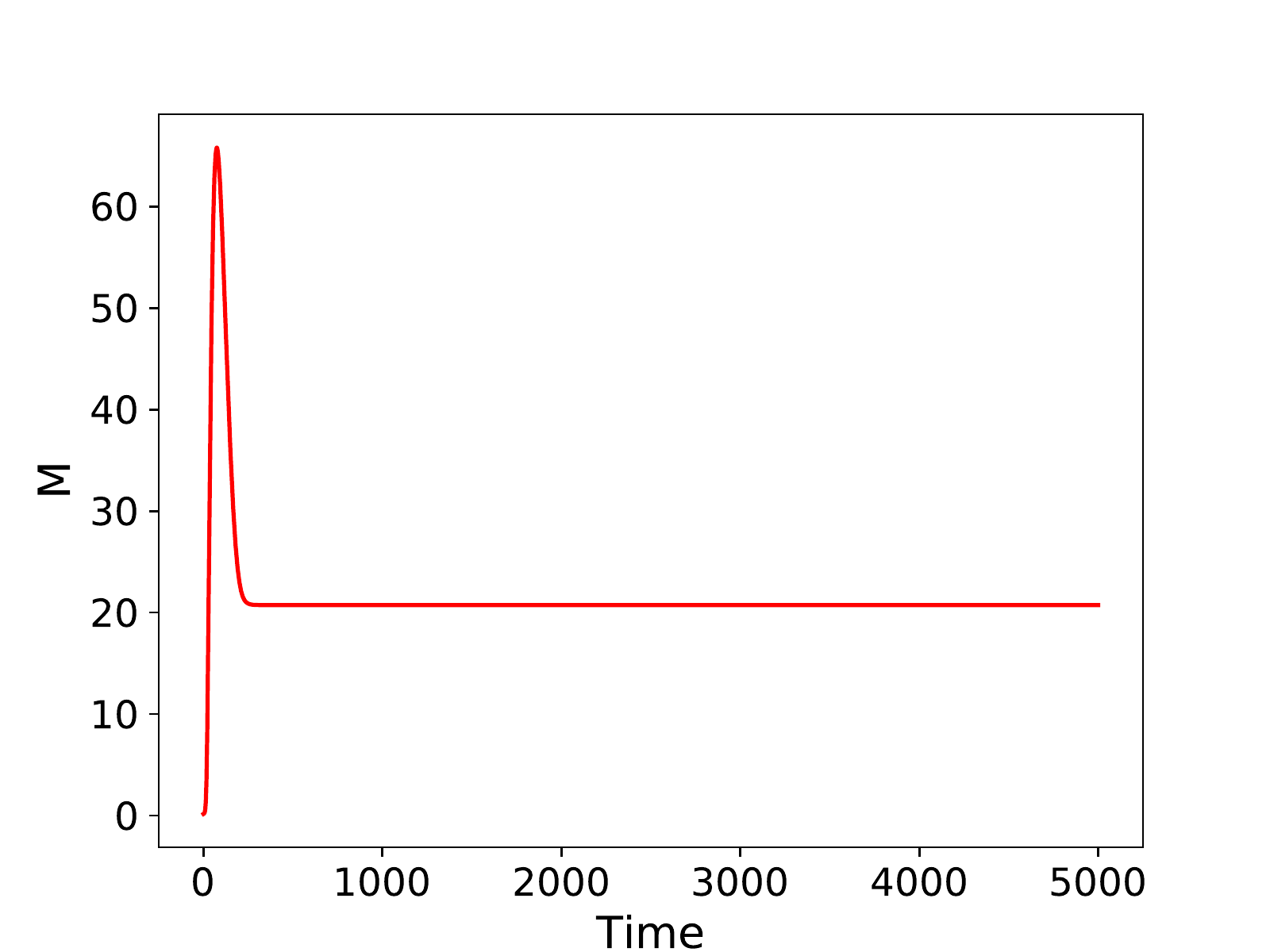}
   \label{fig:M025}
  \end{subfigure}%
  \begin{subfigure}[b]{.5\linewidth}
    \centering
   \caption{}
   \includegraphics[width=\textwidth]{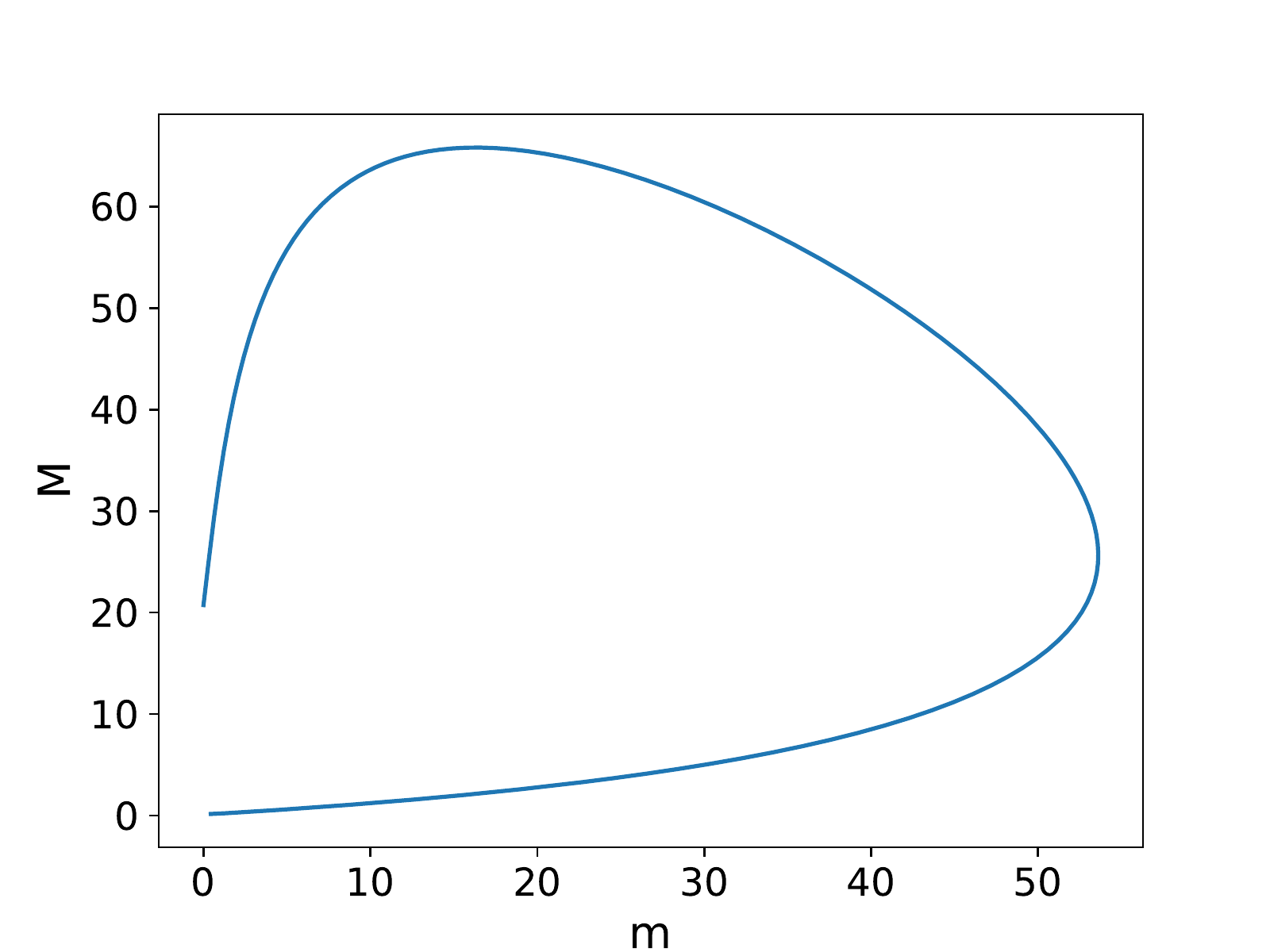}
   \label{fig:PP025}
  \end{subfigure}\\%
  \begin{subfigure}[b]{.5\linewidth}
    \centering
   \caption{}
   \includegraphics[width=\textwidth]{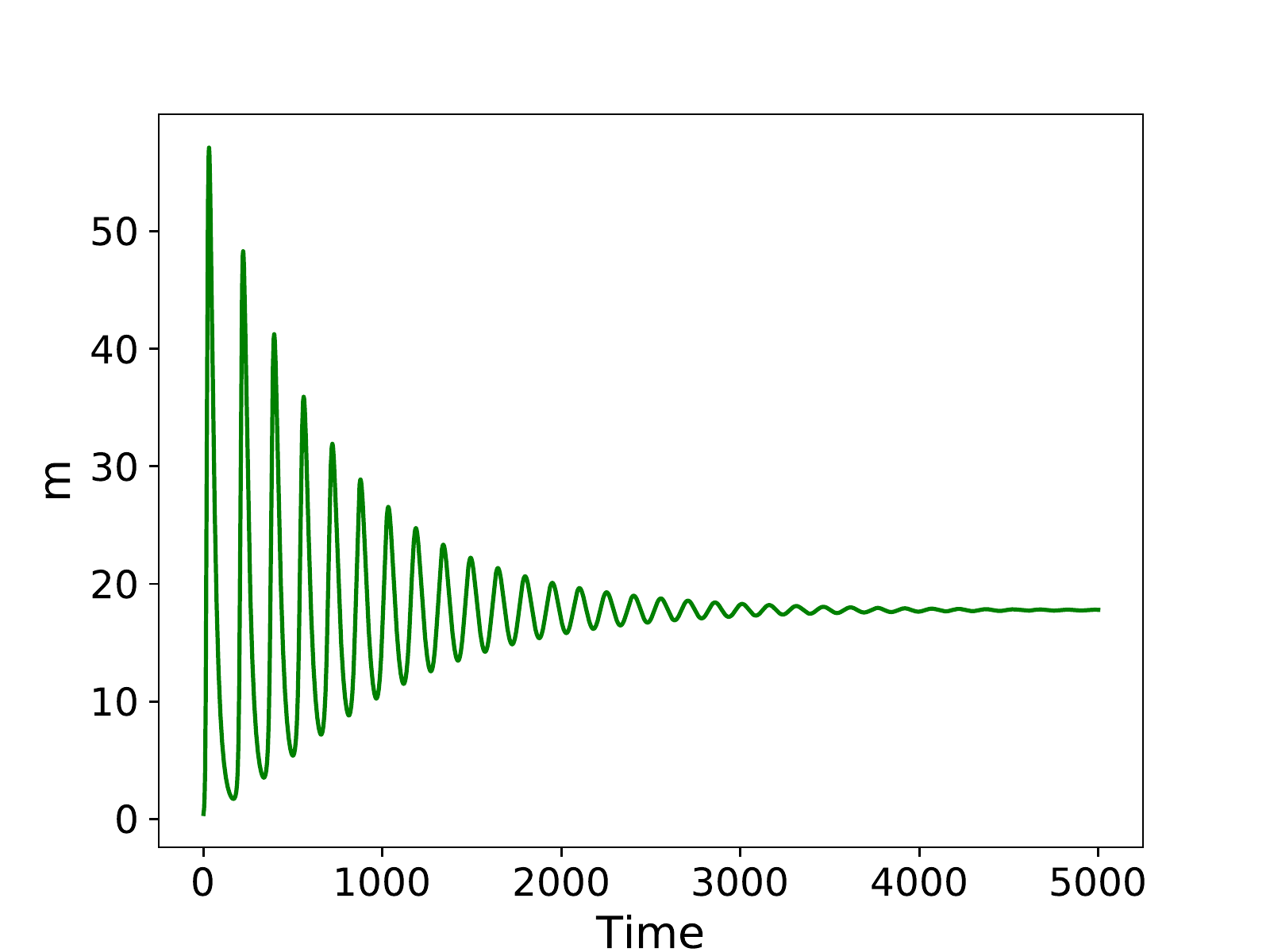}
   \label{fig:m03}
  \end{subfigure}%
  \begin{subfigure}[b]{.5\linewidth}
    \centering
   \caption{}
   \includegraphics[width=\textwidth]{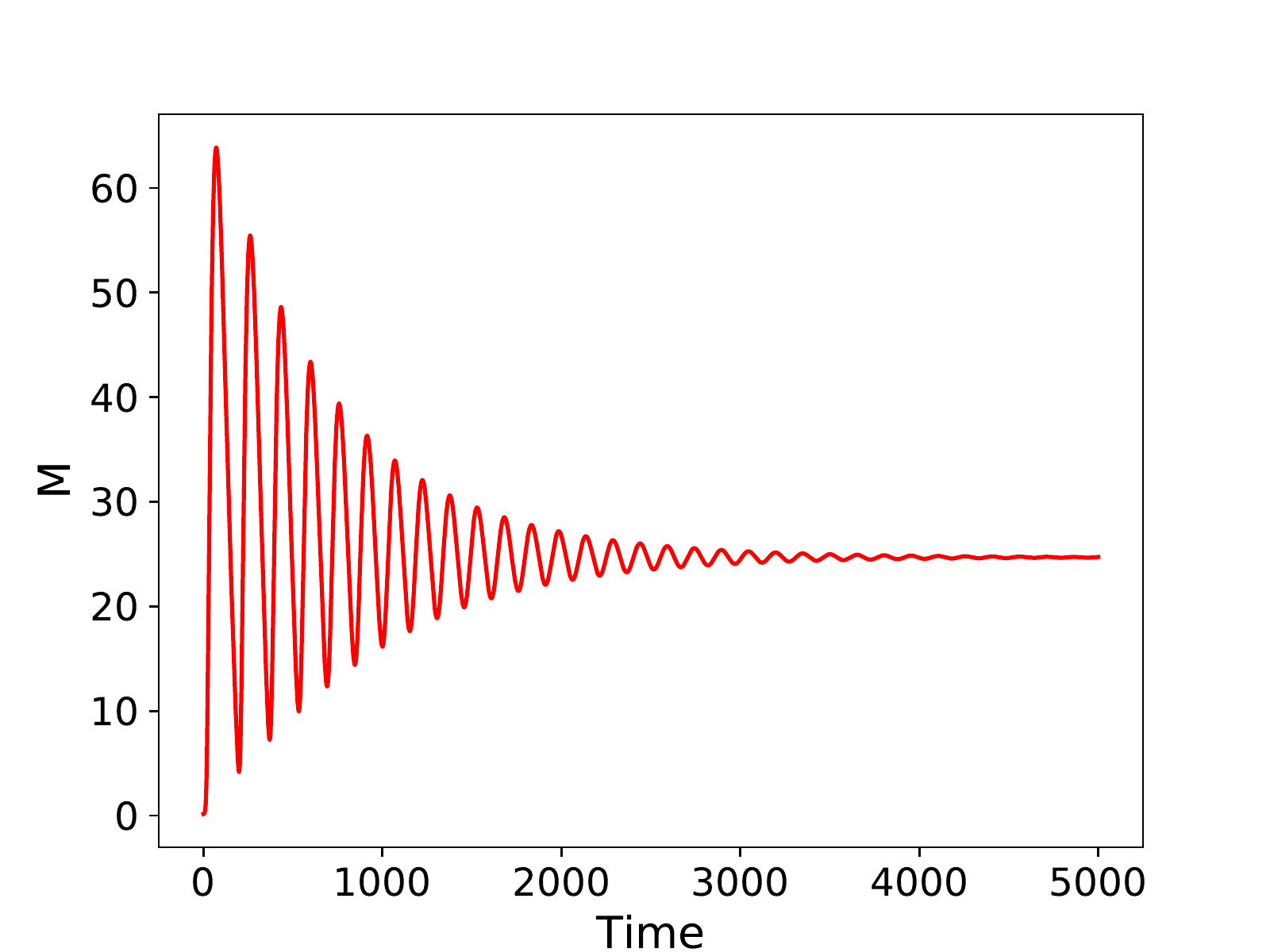}
   \label{fig:M03}
  \end{subfigure}%
  \begin{subfigure}[b]{.5\linewidth}
    \centering
    \caption{}
    \includegraphics[width=\textwidth]{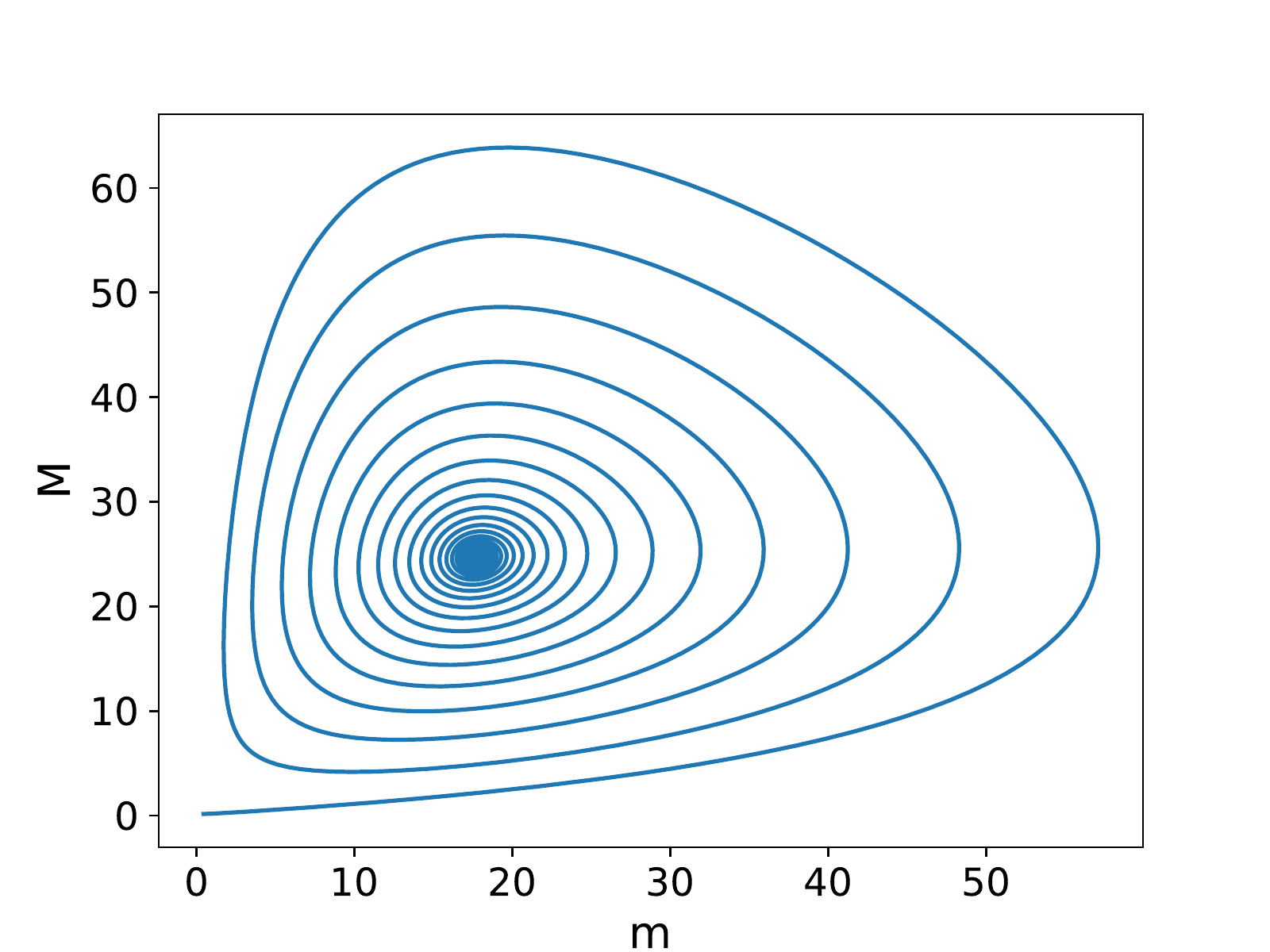}
    \label{fig:PP03}
  \end{subfigure}\\%
  \begin{subfigure}[b]{.5\linewidth}
    \centering
    \caption{}
    \includegraphics[width=\textwidth]{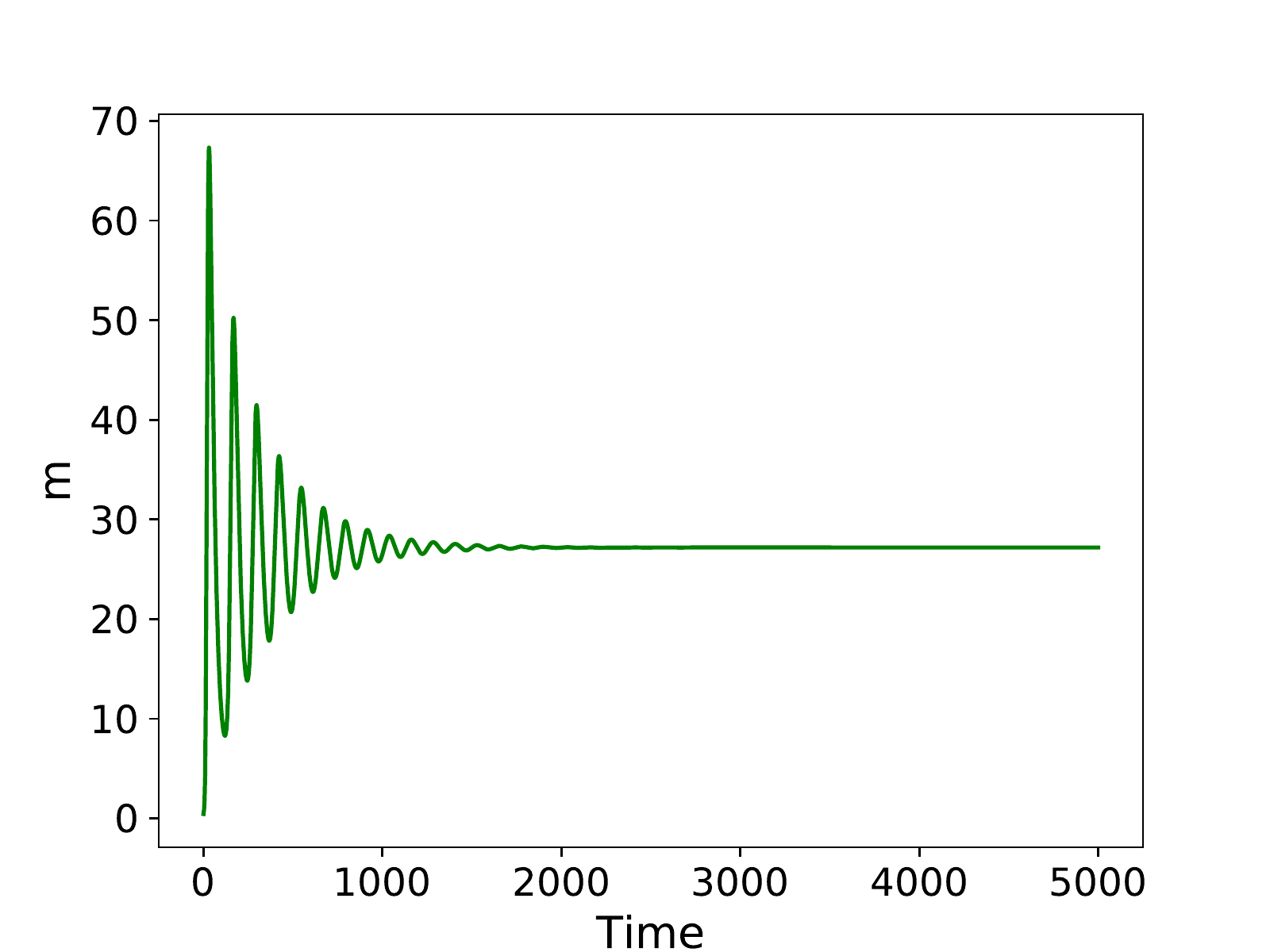}
    \label{fig:m045}
  \end{subfigure}%
  \begin{subfigure}[b]{.5\linewidth}
    \centering
   \caption{}
   \includegraphics[width=\textwidth]{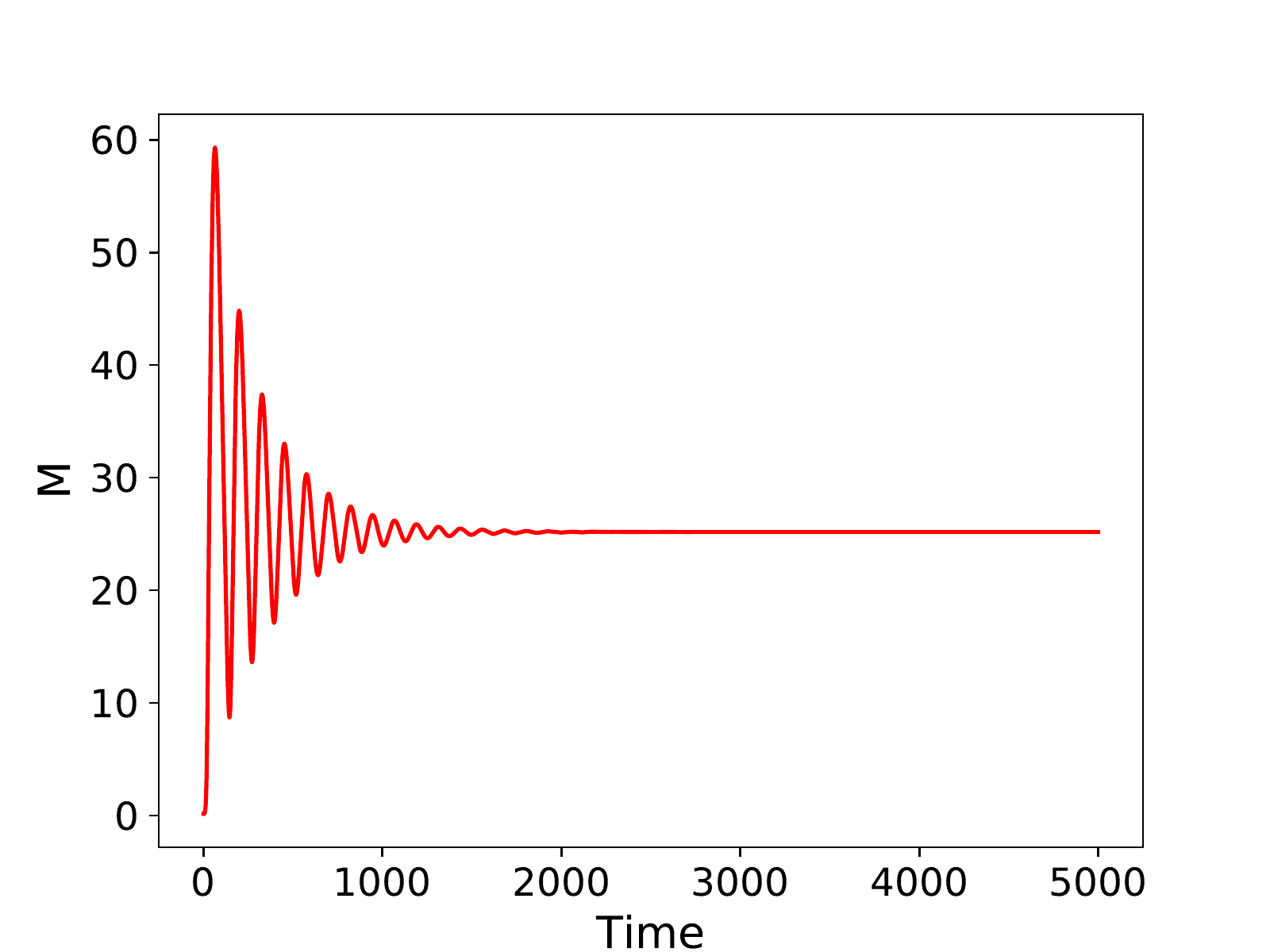}
   \label{fig:M045}
  \end{subfigure}%
  \begin{subfigure}[b]{.5\linewidth}
    \centering
    \caption{}
    \includegraphics[width=\textwidth]{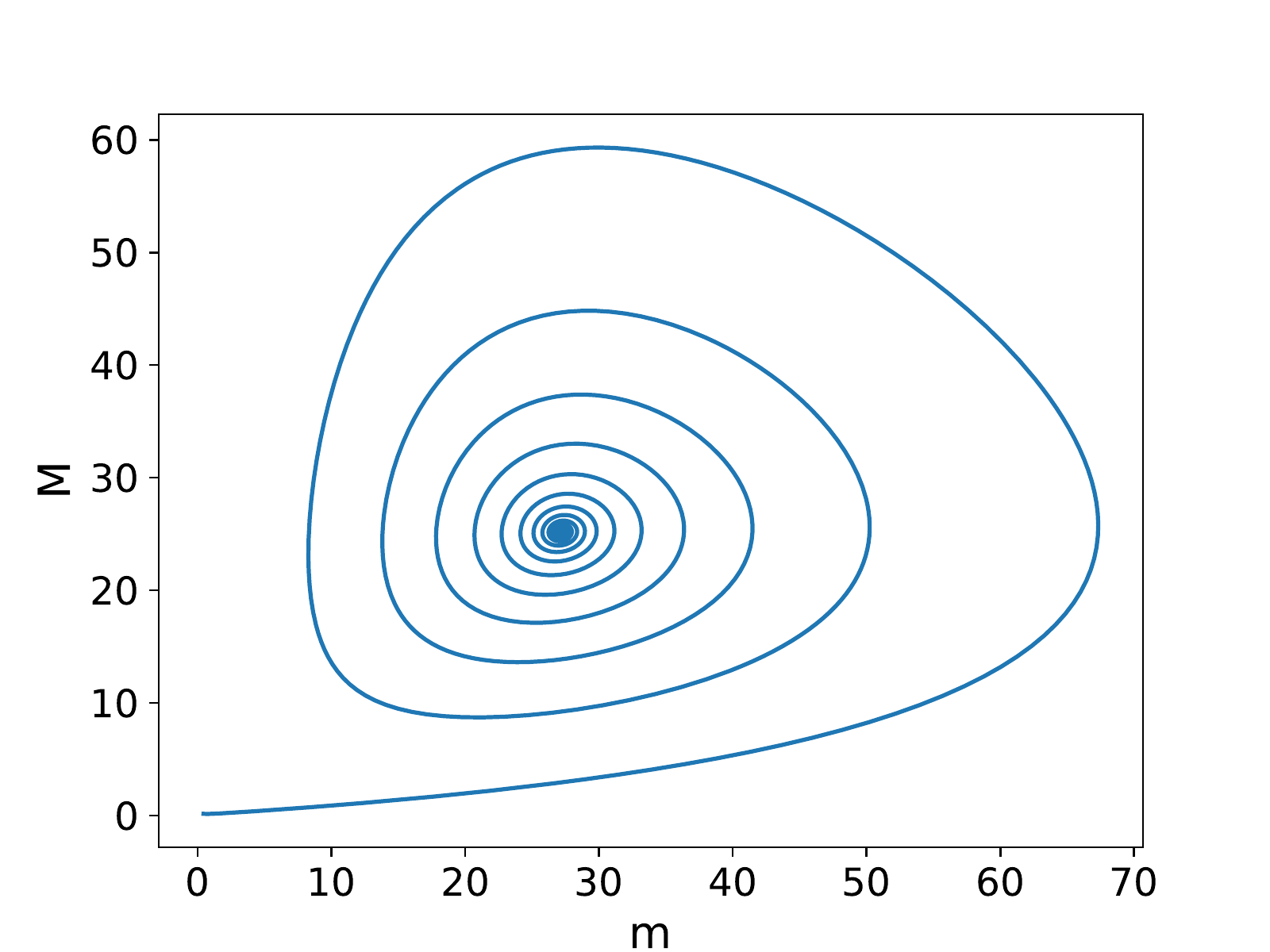}
    \label{fig:pp045}
  \end{subfigure}\\%
    \begin{subfigure}[b]{.5\linewidth}
    \centering
    \caption{}
   \includegraphics[width=\textwidth]{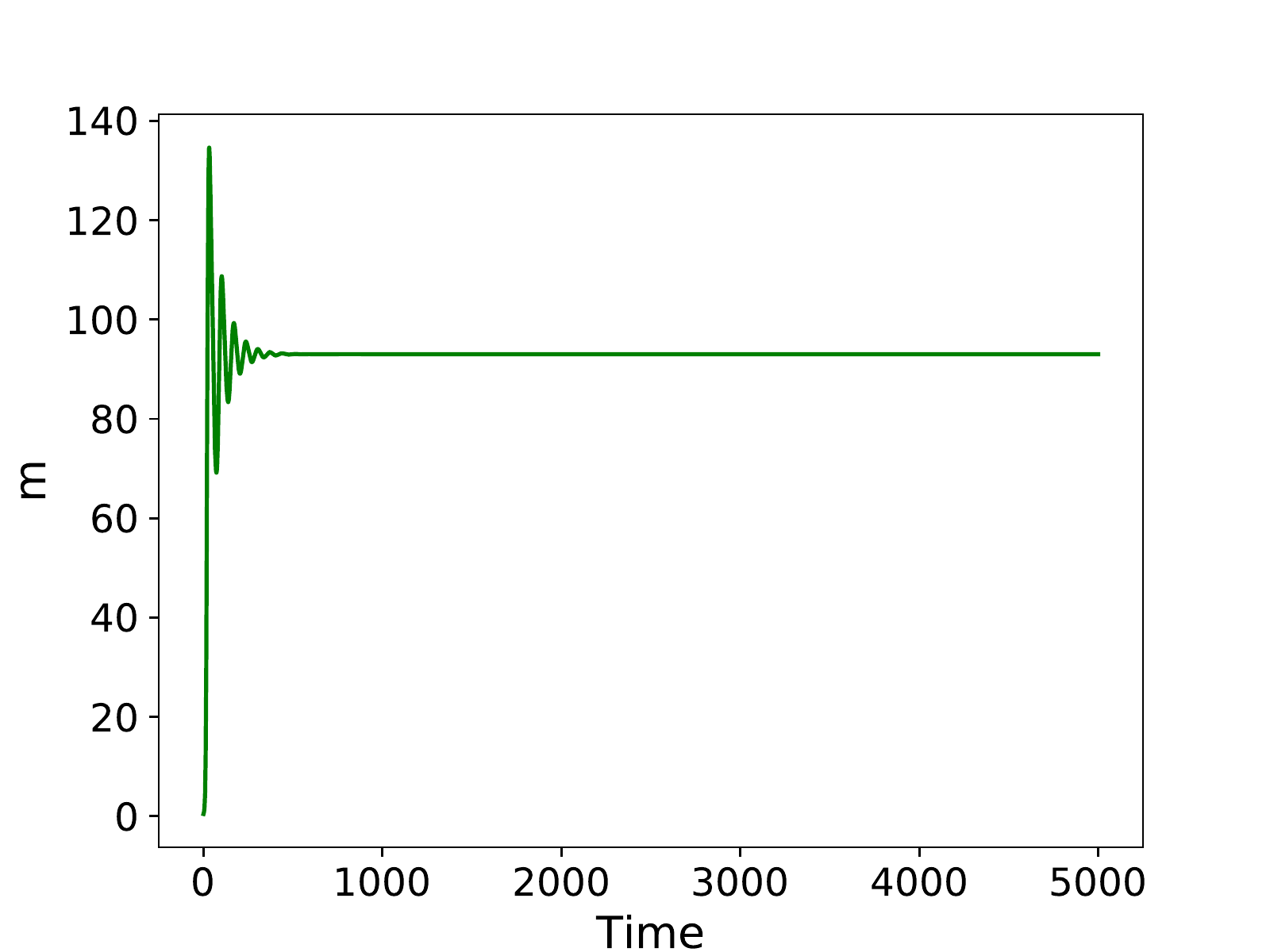}
    \label{fig:m15}
  \end{subfigure}%
   \begin{subfigure}[b]{.5\linewidth}
    \centering
    \caption{}
   \includegraphics[width=\textwidth]{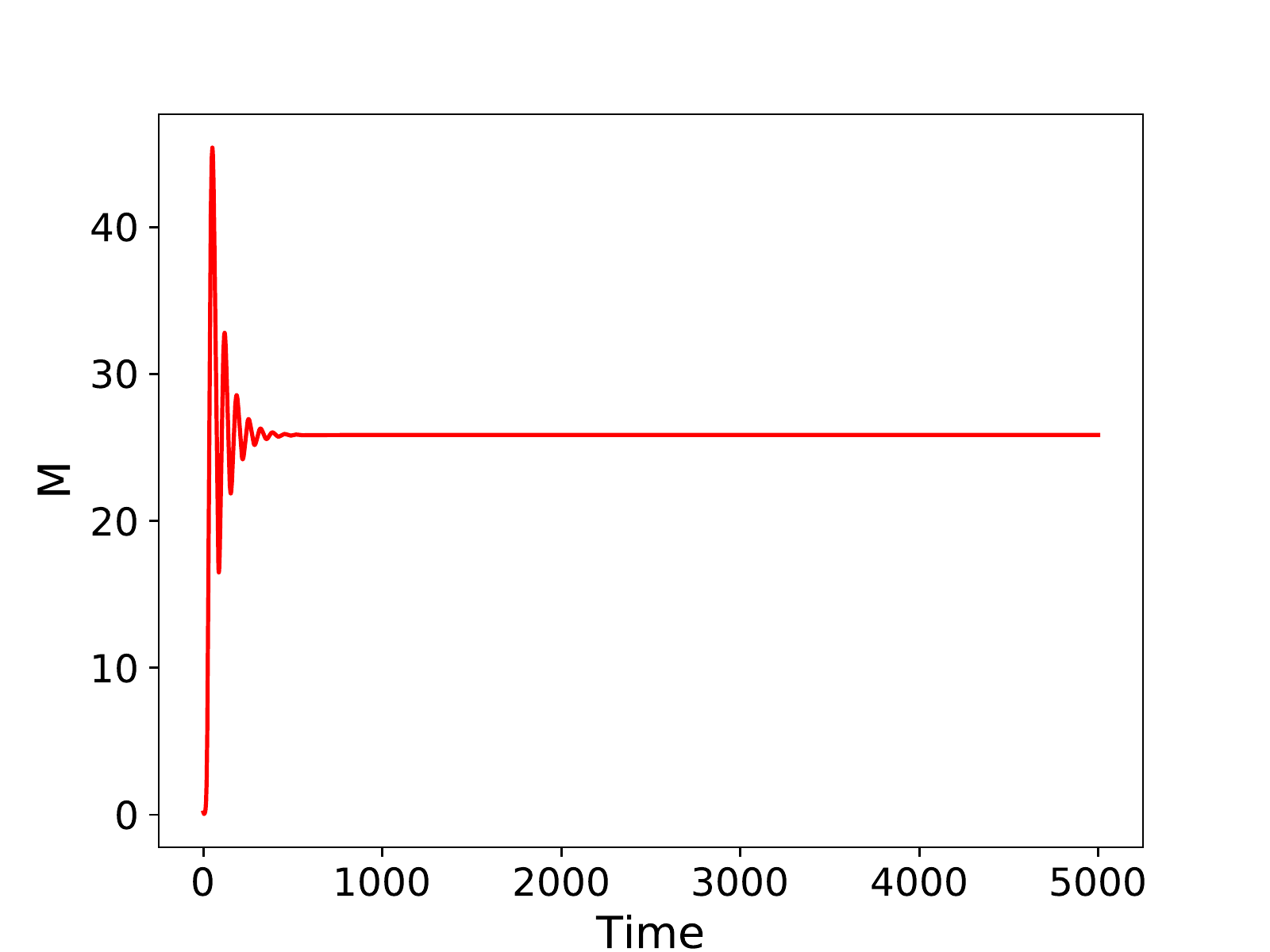}
    \label{fig:M15}
  \end{subfigure}%
  \begin{subfigure}[b]{.5\linewidth}
    \centering
    \caption{}
    \includegraphics[width=\textwidth]{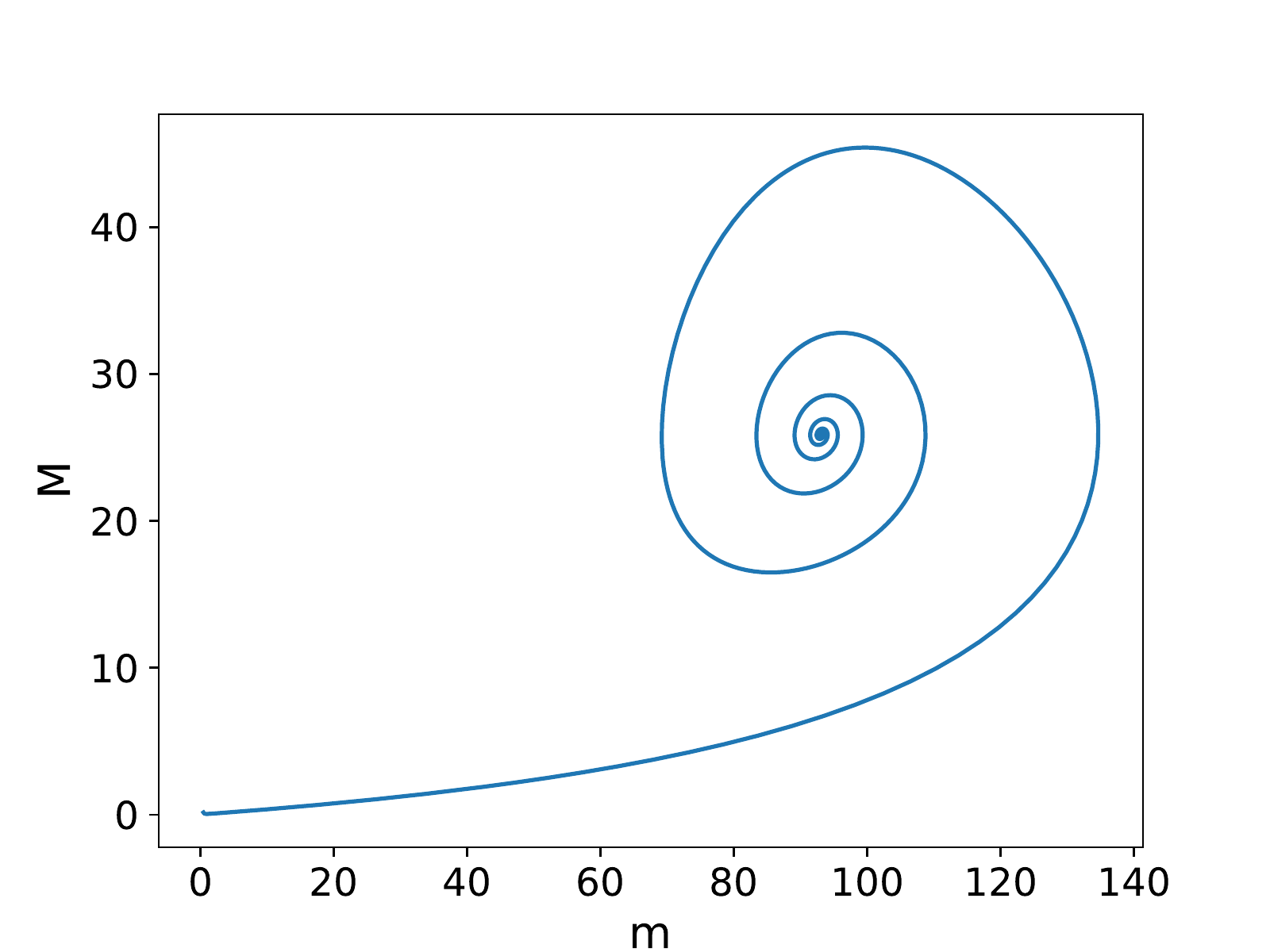}
    \label{fig:PP15}
  \end{subfigure}%
 \caption{Time series plots of monocyte (m) concentration and macrophage concentration (M) for several values of b:  (a)-(b) $b=0.25$; (d)-(e) $b=0.3$; (g)-(h) $b=0.45$; (j)-(k) $b=1.5$ while keeping all other parameter values same as in Table \ref{table:t1}. Their corresponding phase plane are (c), (f), (i) and (l), respectively}
 \label{fig:TT_Pplane}
\end{figure*}
\subsection{\textbf{Linear stability analysis}}\label{ss22}
The  equilibrium points of the system \eqref{eq:5}-\eqref{eq:6} are (i) $E_1= (0,M_1)$ and (ii) $E_2= (m_2,M_2)$, where 
$m_2={\frac {-M_{2}\,cef+M_{2}\,bd-cf}{c \left( M_{2}\,e+d+1 \right) }}$, and $M_2$ is a root of the equation
\begin{equation}
    Ax^2+Bx+C=0, \label{eq10}
\end{equation} 
with $A=bce\sigma+be\epsilon\,\sigma+bce+be\epsilon$, $B=acef+bcd\sigma+bd\epsilon\,\sigma-abd+bcd+bc\sigma+bd\epsilon+b
\epsilon\,\sigma+bc+b\epsilon$, $C=acf$. 

The Jacobian matrix of the system \eqref{eq:5}-\eqref{eq:6}
is,
\begin{eqnarray}
J = \left[ \begin {array}{cc} {\frac {adm \left( 2\,Mef+Mem+dm+2\,f+m
 \right) }{ \left( Mef+Mem+dm+f+m \right) ^{2} \left( 1+\sigma
 \right) }}-\epsilon-c&-{\frac {da{m}^{2}e \left( f+m \right) }{
 \left( 1+\sigma \right)  \left(  \left( Me+d+1 \right) m+f \left( Me+
1 \right)  \right) ^{2}}}\\ \noalign{\medskip}c-{\frac {bMdf \left( Me
+1 \right) }{ \left( e \left( f+m \right) M+f+ \left( d+1 \right) m
 \right) ^{2}}}&-{\frac {dmb \left( dm+f+m \right) }{ \left( Mef+Mem+d
m+f+m \right) ^{2}}}\end {array} \right], \label{10}
\end{eqnarray}
\begin{theorem}
The system \eqref{eq:5}-\eqref{eq:6} is unstable at $E_1$.
\end{theorem}
\begin{proof}
The Jacobian matrix \eqref{10} at $E_1$ is,
\begin{eqnarray*}
J_1=\left[ \begin {array}{cc} -\epsilon-c&0\\ \noalign{\medskip}c-{\frac 
{bM_{1}\,df \left( M_{1}\,e+1 \right) }{ \left( M_{1}\,ef+f \right) ^{
2}}}&0\end {array} \right],
\end{eqnarray*}
The eigenvalues of $J_1$ matrix are $-\epsilon-c$ and $0$. So the eigenvalues, one is negative and the other one is zero. 
\end{proof}

\begin{theorem}
The system \eqref{eq:5}-\eqref{eq:6} is locally asymptotically stable at $E_2$ if 
\begin{enumerate}
     \item  ${\frac {dam_{2}\, \left( 2\,M_{2}\,ef+M_{2}
\,em_{2}+dm_{2}+2\,f+m_{2} \right) }{ \left( M_{2}\,ef+M_{2}\,em_{2}+d
m_{2}+f+m_{2} \right) ^{2} \left( 1+\sigma \right) }} < \epsilon + c $,\\
     \item $c > {\frac {M_{2}\,bdf
 \left( M_{2}\,e+1 \right) }{ \left( e \left( f+m_{2} \right) M_{2}+f+
 \left( d+1 \right) m_{2} \right) ^{2}}}.$
 \end{enumerate}
\end{theorem}

\begin{proof}
The Jacobian matrix \eqref{10} at $E_2$ is:
\begin{eqnarray*}
    J_2 &=& \left[ \begin {array}{cc} {\frac {dam_{2}\, \left( 2\,M_{2}\,ef+M_{2}
\,em_{2}+dm_{2}+2\,f+m_{2} \right) }{ \left( M_{2}\,ef+M_{2}\,em_{2}+d
m_{2}+f+m_{2} \right) ^{2} \left( 1+\sigma \right) }}-\epsilon-c&-{
\frac {da{m_{2}}^{2}e \left( f+m_{2} \right) }{ \left( 1+\sigma
 \right)  \left(  \left( M_{2}\,e+d+1 \right) m_{2}+f \left( M_{2}\,e+
1 \right)  \right) ^{2}}}\\ \noalign{\medskip}c-{\frac {M_{2}\,bdf
 \left( M_{2}\,e+1 \right) }{ \left( e \left( f+m_{2} \right) M_{2}+f+
 \left( d+1 \right) m_{2} \right) ^{2}}}&-{\frac {dm_{2}\,b \left( dm_
{2}+f+m_{2} \right) }{ \left( M_{2}\,ef+M_{2}\,em_{2}+dm_{2}+f+m_{2}
 \right) ^{2}}}\end {array} \right],\\
 &=& \left[\begin {array}{cc}
 \Psi_{11} & \Psi_{12}\\
 \Psi_{21} & \Psi_{22}\end{array}\right].
\end{eqnarray*}
The characteristic equation of this matrix $J_2$ is,
\begin{equation*}
    z^2+A_1z+A_2=0,
\end{equation*}
where, $A_1= -(\Psi_{11}+\Psi_{22})$ and $A_2= (\Psi_{11}\Psi_{22}-\Psi_{12}\Psi_{21})$. According to Routh Hurwitz criterion, the system \eqref{eq:5}-\eqref{eq:6} is locally asymptotically stable at $E_2$ if and only if $A_1 >0$ and $A_2 > 0$. Clearly, $\Psi_{22} < 0$. So, if $\Psi_{11} < 0$, i.e if ${\frac {dam_{2}\, \left( 2\,M_{2}\,ef+M_{2}
\,em_{2}+dm_{2}+2\,f+m_{2} \right) }{ \left( M_{2}\,ef+M_{2}\,em_{2}+d
m_{2}+f+m_{2} \right) ^{2} \left( 1+\sigma \right) }} < \epsilon + c$, then $A_1 > 0$. Now, $\Psi_{22}$ and $\Psi_{12}$ are negative always. So whenever $\Psi_{11} < 0$, if $\Psi_{21} >0$ then we can obtain $A_2 >0$. Now when $c > {\frac {M_{2}\,bdf
 \left( M_{2}\,e+1 \right) }{ \left( e \left( f+m_{2} \right) M_{2}+f+
 \left( d+1 \right) m_{2} \right) ^{2}}}$ holds then we have $\Psi_{21} >0$. Thus the conditions for local asymptotic stability at the nonzero equilibrium point $E_2$ for the sytem \eqref{eq:5}-\eqref{eq:6} are 
 \begin{enumerate}
     \item  ${\frac {dam_{2}\, \left( 2\,M_{2}\,ef+M_{2}
\,em_{2}+dm_{2}+2\,f+m_{2} \right) }{ \left( M_{2}\,ef+M_{2}\,em_{2}+d
m_{2}+f+m_{2} \right) ^{2} \left( 1+\sigma \right) }} < \epsilon + c $,\\
     \item $c > {\frac {M_{2}\,bdf
 \left( M_{2}\,e+1 \right) }{ \left( e \left( f+m_{2} \right) M_{2}+f+
 \left( d+1 \right) m_{2} \right) ^{2}}}.$
 \end{enumerate}
\end{proof}
Stability diagrams for different values of parameter $b$ while keeping all other parameter values same as in table \ref{table:t1} is provided Fig. \ref{fig:TT_Pplane}.

Clearly, equation \eqref{eq10} has two roots lets say $M_{p}$ and $M_{n}$. Then let us consider the two fixed points as $(m_{p},M_{p})$ and $(m_{n},M_{n})$ where,
\begin{align}
m_{p} &= \frac{-M_{p}cef + M_{p}bd -cf}{c_{2}(M_{p}e + d + 1)},\\
\label{eq:posE2M}
M_{p} &= \frac{-B + \sqrt{B^2-4AC}}{2A},\\
m_{n} &= \frac{-M_{n}cef + M_{n}bd -cf}{c_{2}(M_{n}e + d + 1)},\\
M_{n} &= \frac{-B - \sqrt{B^2-4AC}}{2A},
\label{eq:negE2}
\end{align}
\textcolor{black}{Below we have provided a two parameter $b-d$ stability diagram. In panel (a) of Fig. \ref{fig:MSNd3}, we observe that larger value of $b$ makes the equilibrium point $(m_{p},M_{p})$ to be stable (marked in blue), while the lower values of $b$ makes the equilibrium point unstable (marked in red) and also its nonexistence (marked in grey) in an alternative manner. In panel (b), we see that  the equilibrium point $(m_{n},M_{n})$ exhibits saddle behavior (marked in yellow) for larger values of $b,d$. Lower values of the parameters makes the equilibrium point $(m_{p},M_{p}), (m_{n},M_{n})$ non-existent.}

\begin{figure}[!htbp]
\begin{center}
\includegraphics[width=1\textwidth]{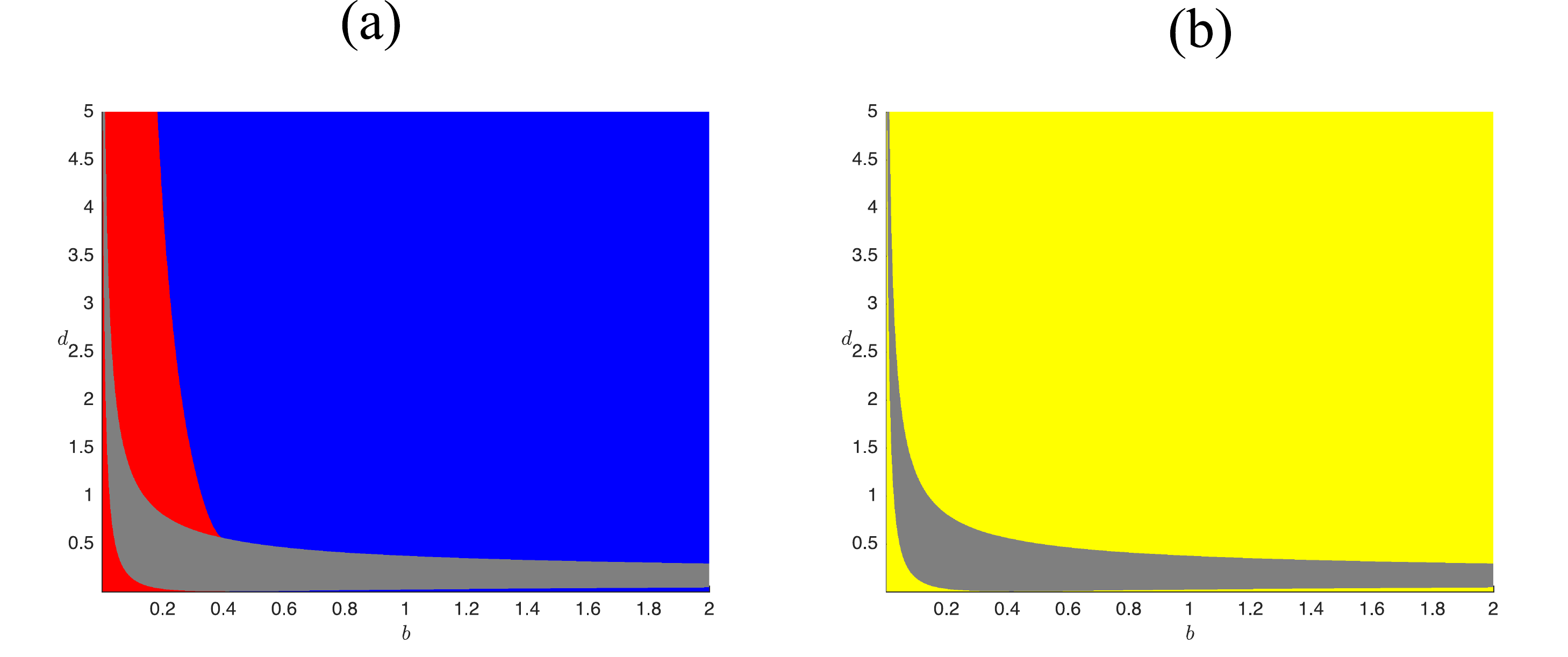}
\end{center}
\caption{(a) Two parameter $b-d$ stability plane. Blue colour denotes the region where the fixed point $(m_{p},M_{p})$ is stable. Red colour denotes the region where the fixed point $(m_{p},M_{p})$ is unstable. The grey region denotes the non-existence of fixed point $E_{2}$. The remaining parameters fixed as $a=1,d=0.9,\sigma=1,\epsilon=0.01,c=0.05,f=1,e=1$. (b) Two parameter $b-d$ stability plane.  The grey region denotes the non-existence of fixed point $E_{2}$. The yellow colour denotes the saddle nature of fixed point $(m_{n},M_{n})$. The remaining parameters fixed as $a=1,d=0.9,\sigma=1,\epsilon=0.01,c=0.05,f=1,e=1$} 
   \label{fig:MSNd3}
\end{figure}

To showcase the Hopf bifurcation, say we fix $d=3$ and continuate the eigenvalues of the two fixed points obtained from \eqref{eq:posE2M}-\eqref{eq:negE2} with parameter $b$. Here $\lambda_{p},\lambda_{n}$ denote the eigenvalues of the the fixed points $(m_{p},M_{p})$ and $(m_{n},M_{n})$ respectively. As we observe in Fig. \ref{fig:Hopfd3}, the real part of $\lambda_{p}$ crosses $0$ at $b = 0.2308$ where we detect the Hopf bifurcation. We also detect saddle-node bifurcation as we vary parameter $b$. In Fig. \ref{fig:SNbifurcat} (a) and (b), we continuate the first ($m$) and second component ($M$) of the fixed point with parameter $b$. The red and blue colour denote fixed points $(m_{p},M_{p})$ and $(m_{n},M_{n})$ respectively. As parameter $b$ is decreased, $m_{p}, m_{n}$ come closer and collide at $b = 0.03013$ and then we observe only a single component emerging out for $b < 0.03013$, illustrating  a saddle-node bifurcation.

 \begin{figure}[htbp]
   \centering
    \includegraphics[width=0.5\textwidth]{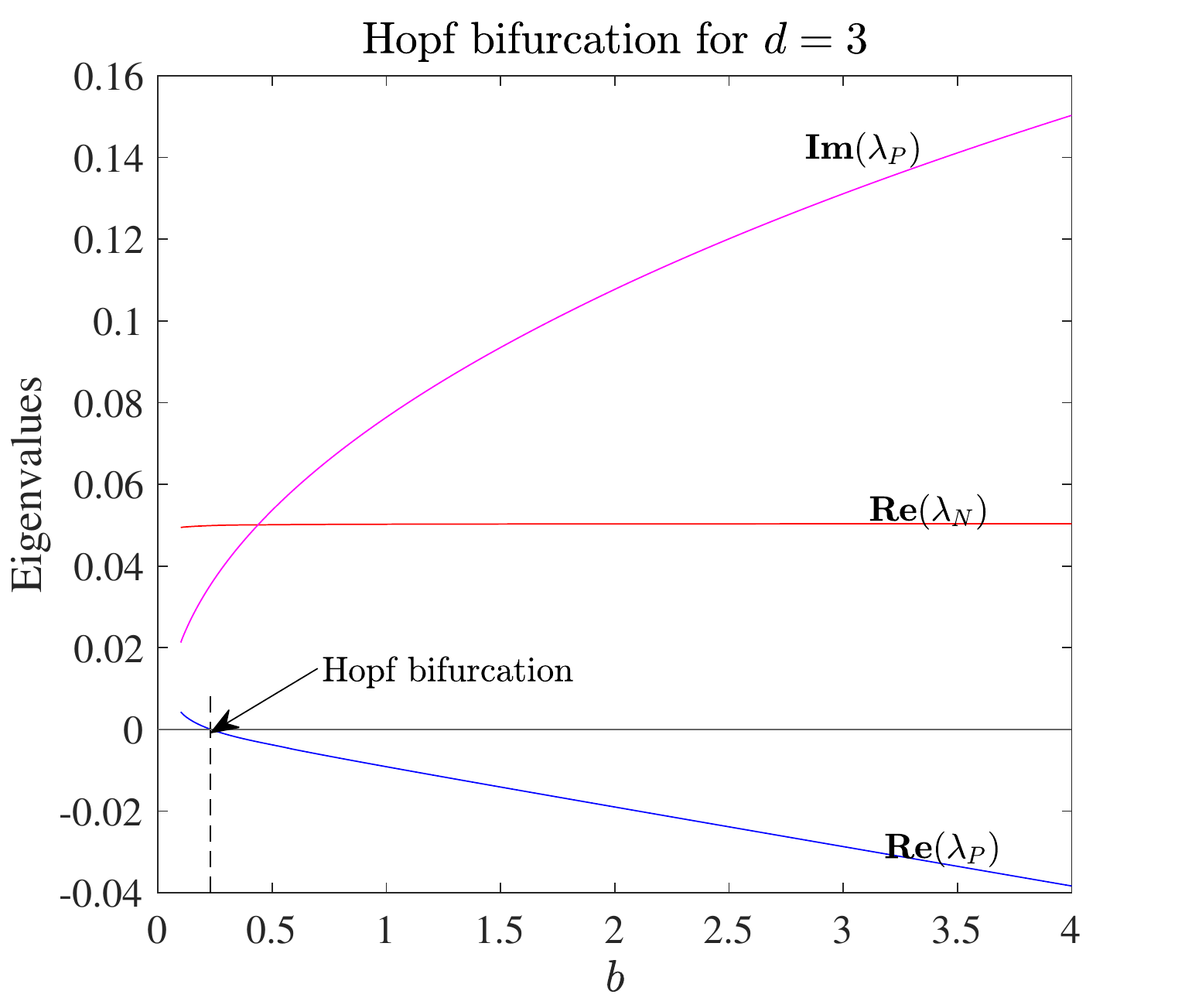}
    \caption{Hopf bifurcation detection for $d=3$ as parameter $b$ varies. The red curve denotes the variation of the eigenvalues of the Jacobian evaluated at the equilibrium point $(m_{n},M_{n})$. The magenta curve shows the variation of the imaginary part of the eigenvalue of the Jacobian at $(m_{p},M_{p})$. The blue curve denotes the eigenvalue of the Jacobian at $(m_{p},M_{p})$. The real part of the eigenvalue crosses zero denoting a Hopf bifurcation at $b=0.2308$}
    \label{fig:Hopfd3}
  \end{figure}

  \begin{figure}[!htbp]
\begin{center}
\includegraphics[width=1\textwidth]{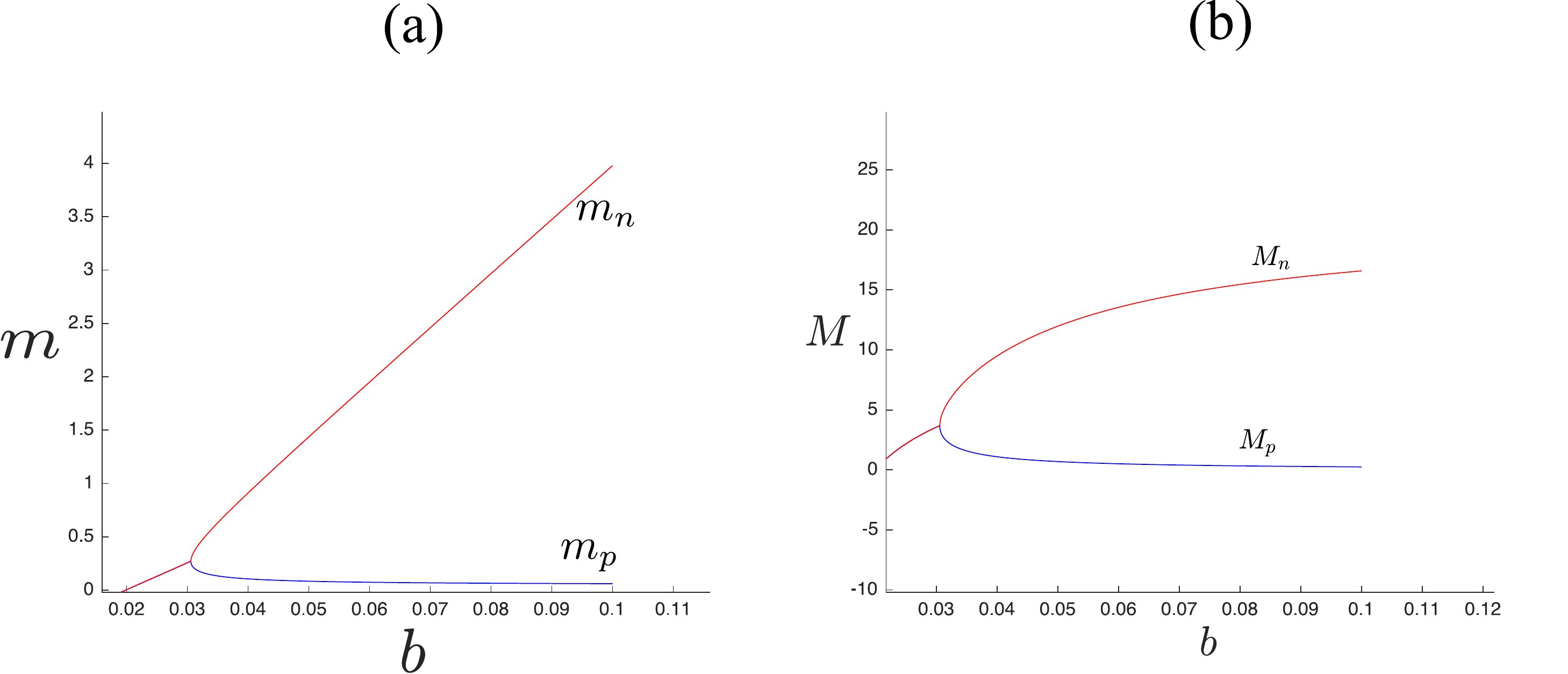}
\end{center}
\caption{(a) Saddle-node bifurcation for $d=3$ as parameter $b$ varies. The red curve denotes the first component of the fixed point $(m_{n},M_{n})$ of type $E_{2}$ and the blue curve denotes the first component of the fixed point $(m_{p},M_{p})$. (b) Saddle-node bifurcation for $d=3$ as parameter $b$ varies. The red curve denotes the second component of the fixed point $(m_{n},M_{n})$ of type $E_{2}$ and the blue curve denotes the second component of the fixed point $(m_{p},M_{p})$} 
   \label{fig:SNbifurcat}
\end{figure}

\subsection{\textbf{Bifurcation analysis of the reduced model}}\label{sec:bifurcation}
In this section, we present in detail a numerical bifurcation of the reduced model \eqref{eq:5}--\eqref{eq:6}. We discuss the bifurcation scenarios as the parameters are varied from their values in Table \ref{table:t1}.Here, the ingestion rate of oxidized LDL by macrophages $b$ and the intake rate of oxidized LDL concentration $d$ are considered as bifurcation parameters. The parameter $b$ represents how fast oxidized LDL particles can be removed from the affected region inside the intima and this is the prime reason to choose $b$ as one of the bifurcating parameter. As in atherosclerotic plaque formation, the rate at which LDL particle enters into the intima play a major role hence the parameter $d$ is considered as another bifurcating parameters \citep{frisdal2011interleukin}. The bifurcation diagrams were produced via numerical continuation software MATCONT \citep{Matcont2003}. The bifurcation diagram for the concentration of monocyte $m$ as a function of the conversion rate of macrophages $b$ is shown in Fig.~\ref{fig:1pard=3}. It is divided into three regions. In region I, $0.03059<b<0.2314$, there exist two equilibria: a stable (upper branch) and an unstable (lower branch) solution. The equilibria collide and annihilate in a saddle-node bifurcation, denoted ${\rm LP}$, at $b=0.03059$. Increasing $b$, the system undergoes a Hopf bifurcation, denoted ${\rm HB}$, along with the upper branch equilibrium at $b=0.2314$. The first Lyapunov coefficient $\alpha=3.7483\times10^{-6}$, since $\alpha$ is positive, then the HB is subcritical type, and the periodic oscillation that appears at the ${\rm HB}$ point is unstable. The unstable periodic oscillation terminates at homoclinic bifurcation with period=176.327, see Fig.~\ref{fig:limitcycle}. To the right of the HB, the unstable equilibrium (upper branch) gains stability, so in region II there are two stable equilibria. Thus in this region exists bistable phenomenon. Further increasing $b$, there exists only one stable equilibrium. A neutral saddle, denoted ${\rm NS}$, appears at $b=0.43656$. At the NS point, the sum of the eigenvalues is zero, meaning it is a hyperbolic saddle and is not a bifurcation point.
\begin{figure}[htbp]
\centering
  \begin{subfigure}[b]{.6\linewidth}
    \centering
   \caption{}
   \includegraphics[width=\textwidth]{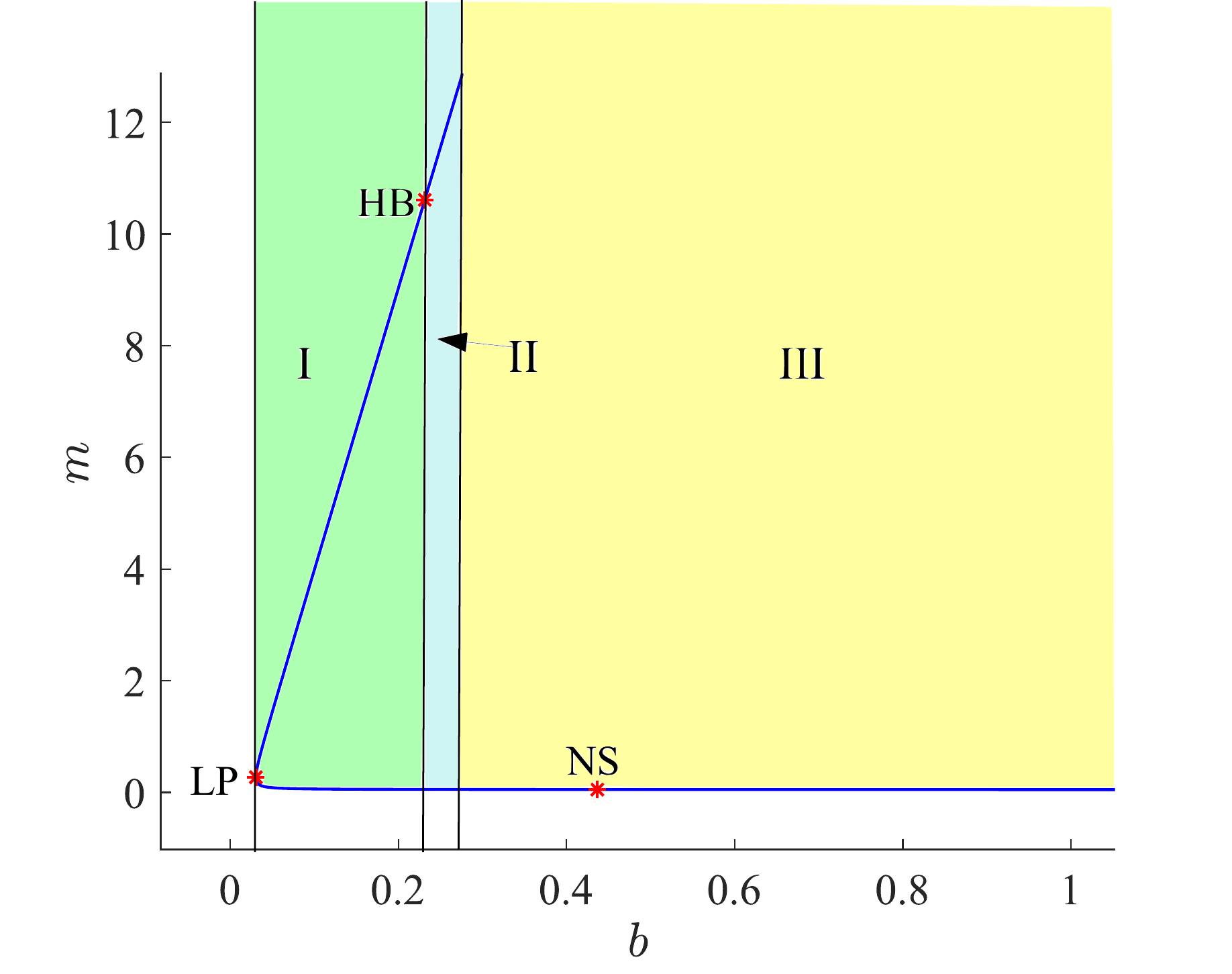}
   \label{fig:1pard=3}
  \end{subfigure}%
   \begin{subfigure}[b]{.7\linewidth}
    \centering
   \caption{}
   \includegraphics[width=\textwidth]{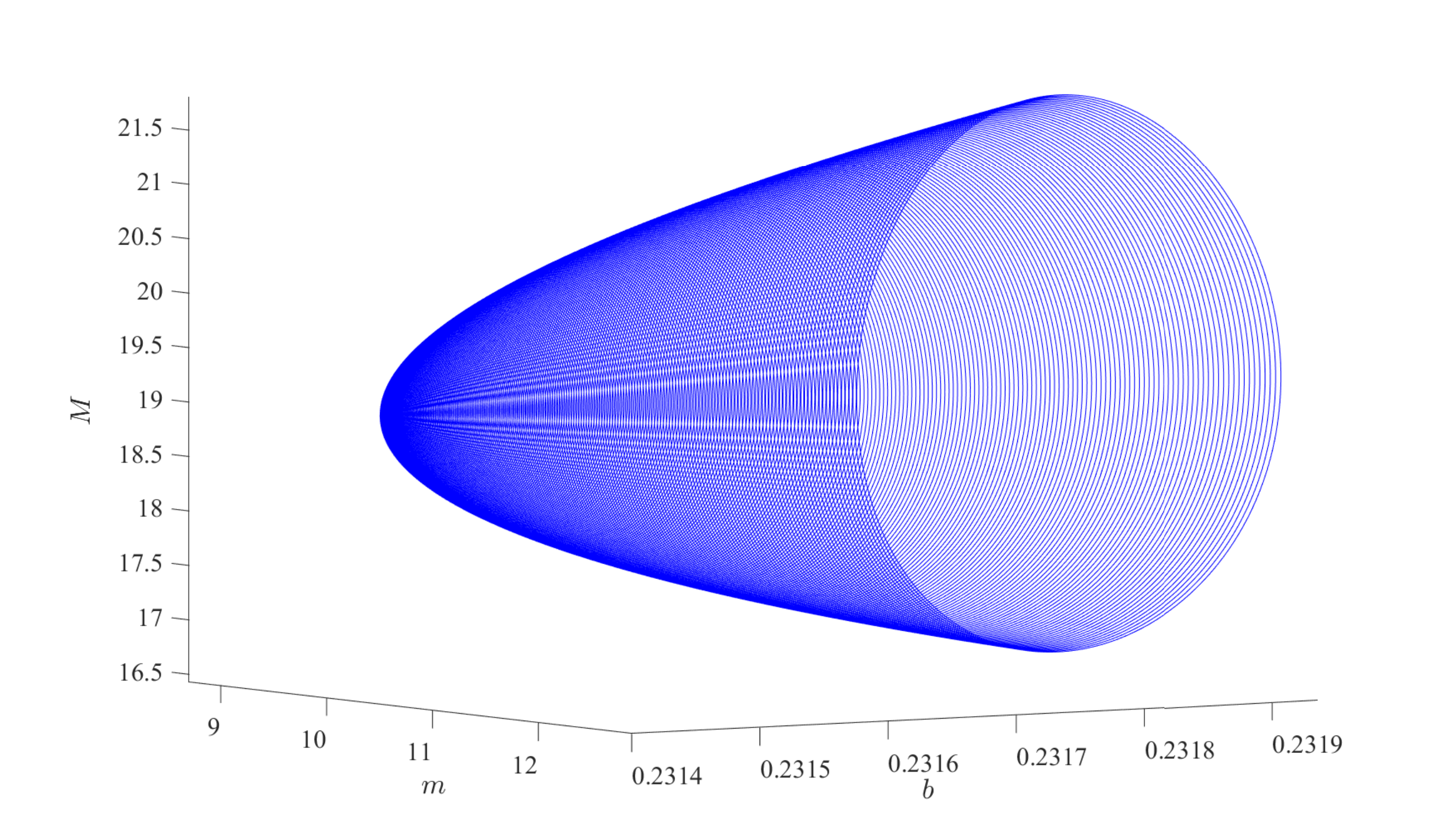}
   \label{fig:limitcycle}
  \end{subfigure}%
    \caption{(a) One-parameter bifurcation diagram of \eqref{eq:5}--\eqref{eq:6} varying $b$. other parameter values as in Table \ref{table:t1}. The bifurcations are labelled as follows: SN: saddle-node bifurcation and HB: Hopf bifurcation}
    \label{fig:One-parameter bifurcation}
  \end{figure}
  
 Next, we compute two-parameter bifurcation analysis by varying parameters $b$ and $d$. Here we consider two bifurcation set: $e=1$ and $e=5$, other system parameters are fixed as show in Table~\ref{table:t1}. The codimension-2 bifurcation diagram in $(b,d)$-plane when $e=1$ is shown in Fig.~\ref{fig:2par1}. The black, magenta, and red curves correspond to the loci of the saddle-node (LP), Hopf (HB), and neutral-saddle (NS) bifurcations, respectively. For extremely low values of $d$, there exists only the saddle-node curve meaning the system has only one saddle bifurcation (see Fig.~\ref{fig:2parzoom}). Increasing $d$, appears a Bogdanov-Takens (BT) bifurcation at $(b,d)=(0.42,0.543)$, this is a codimension-2 bifurcation. The loci of the Hopf and neutral-saddle bifurcations emanate from the BT point. At the BT point, the curve of the saddle-node bifurcation coincide tangentially with the curves of the Hopf and neutral-saddle bifurcations. Thus, besides the saddle-node bifurcation that already exist appears a Hopf bifurcation and neutral saddle bifurcation for values of $d$ above the BT bifurcation. Fig.~\ref{fig:2par1_1} is divided into three regions based on different types of dynamical behaviors. Each region is assigned number I, II, and III. See Table~\ref{table:codim2sym}.
  \begin{figure}[htbp]
      \centering
       \begin{subfigure}[b]{.6\linewidth}
    \centering
    \caption{}
   \includegraphics[width=\textwidth]{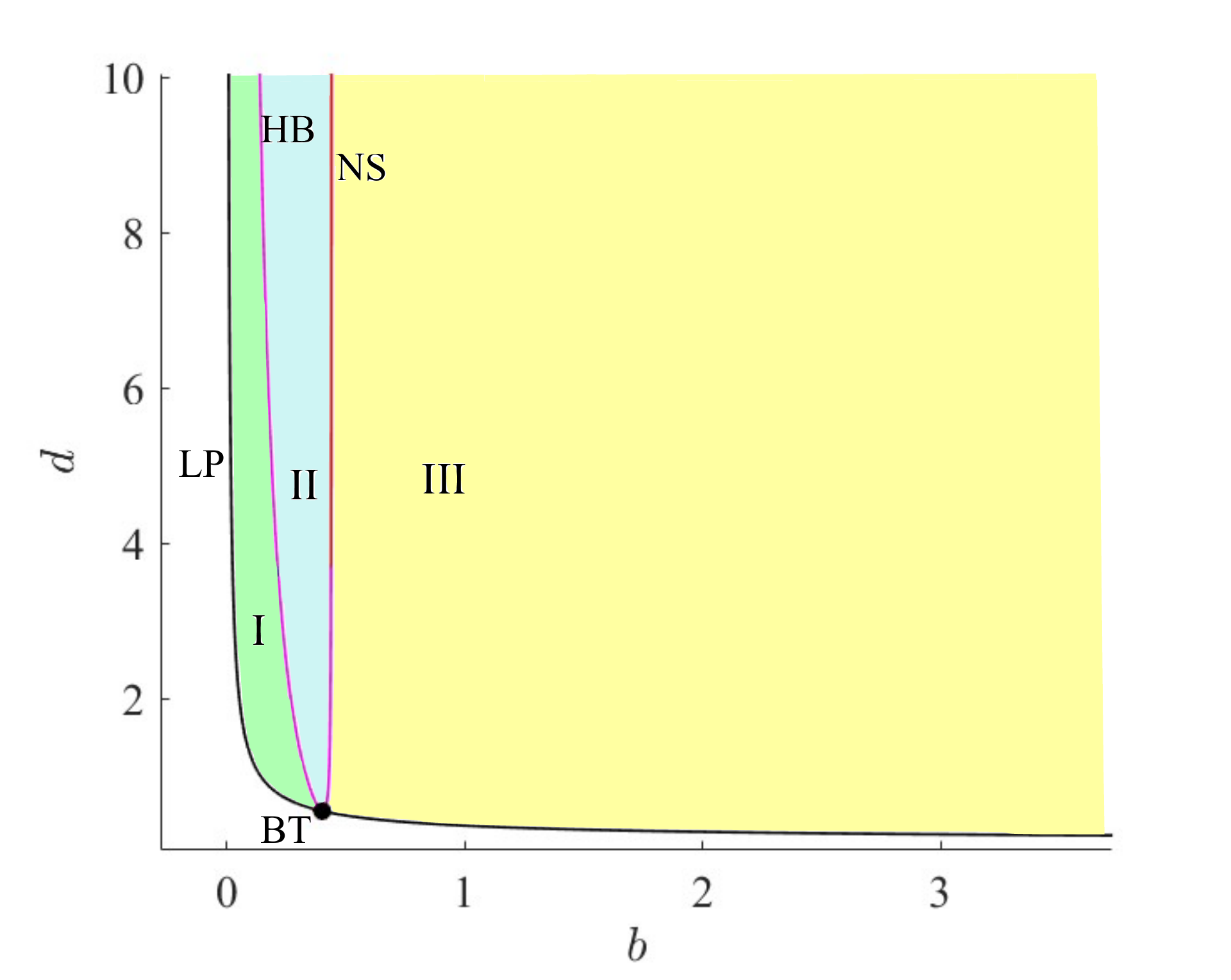}
    \label{fig:2par1_1}
  \end{subfigure}%
  \begin{subfigure}[b]{.6\linewidth}
    \centering
    \caption{}
   \includegraphics[width=\textwidth]{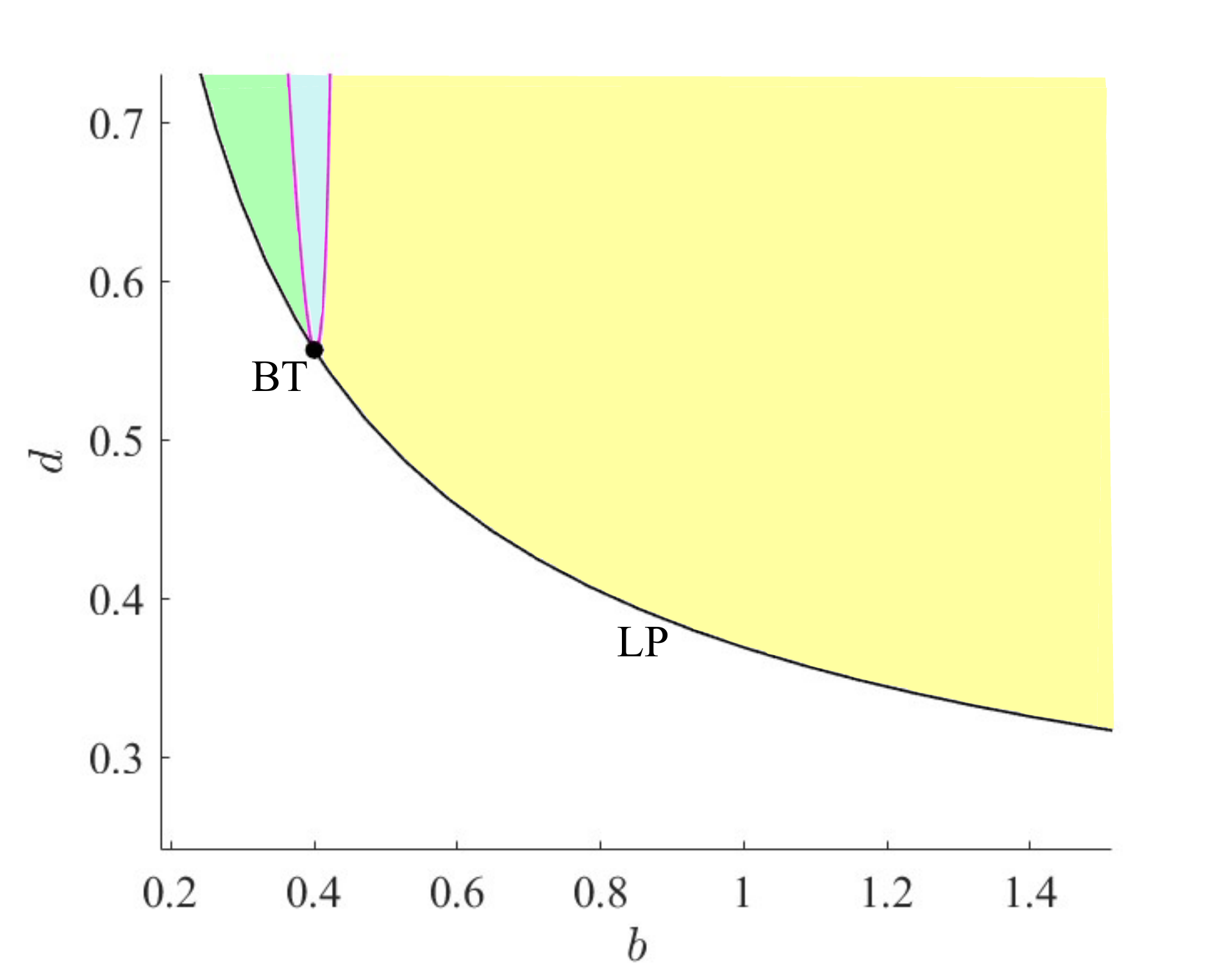}
    \label{fig:2parzoom}
  \end{subfigure}%
  \caption{Two parameter bifurcation diagram of \eqref{eq:5}--\eqref{eq:6} in $(b,d)$-plane with (a) $e=1$ and (b) $e=5$. Other parameter values as in Table \ref{table:t1}. The black and magenta curves are the loci of codimension-one bifurcations labelled as follows: SN: saddle-node bifurcation and HB: Hopf bifurcation}
  \label{fig:2par1}
  \end{figure}


 \begin{table}
\centering
\caption{Summary of the dynamics in the three regions in Fig.~\ref{fig:2par1_1}}
\label{table:codim2sym}
\begin{tabular}{|c|c|}
\hline
Region & Dynamical behaviour \\
\hline
I & One stable equilibrium, one unstable equilibrium, one unstable limit cycle.\\
 \hline
 II&  Two stable equilibria, one unstable limit cycle\\
 \hline
 III & One stable equilibrium, no limit cycle \\
 \hline
\end{tabular}
\label{table:t2}
\end{table}

Lastly, we set $e=5$ with other parameters fixed as in Table~\ref{table:t1}. Fig.~\ref{fig:2par2} shows a bifurcation diagram in $(b,d)$-plane. The bifurcation structure is similar to when $e=1$ except for an additional co-dimension one bifurcation point detected. For sufficiently low values of $d$, there is only a saddle-node LP curve. As we increase the value of $d$ a BT bifurcation is detected along the LP curve at $(b,d)=(0.6489,1.063)$. From the theory of Bogdanov-Takens bifurcation \citep{KuznetsovY.A.1995ElementsTheory}, the locus of the Hopf bifurcation HB emanates from the BT point and is tangent to the LP curve at this point. As the value of $d$ increases further, a generalized Hopf GH bifurcation is observed at $(d,b)=(0.5004,3.146)$, this is a codimension-two bifurcation where the HB changes from subcritical to supercritical. Further increasing the value of $d$ above the GH bifurcation, there exist only the loci of the saddle-node and HB bifurcations.
\begin{figure}[htbp]
  \centering
   \begin{subfigure}[b]{.6\linewidth}
    \centering
   \includegraphics[width=\textwidth]{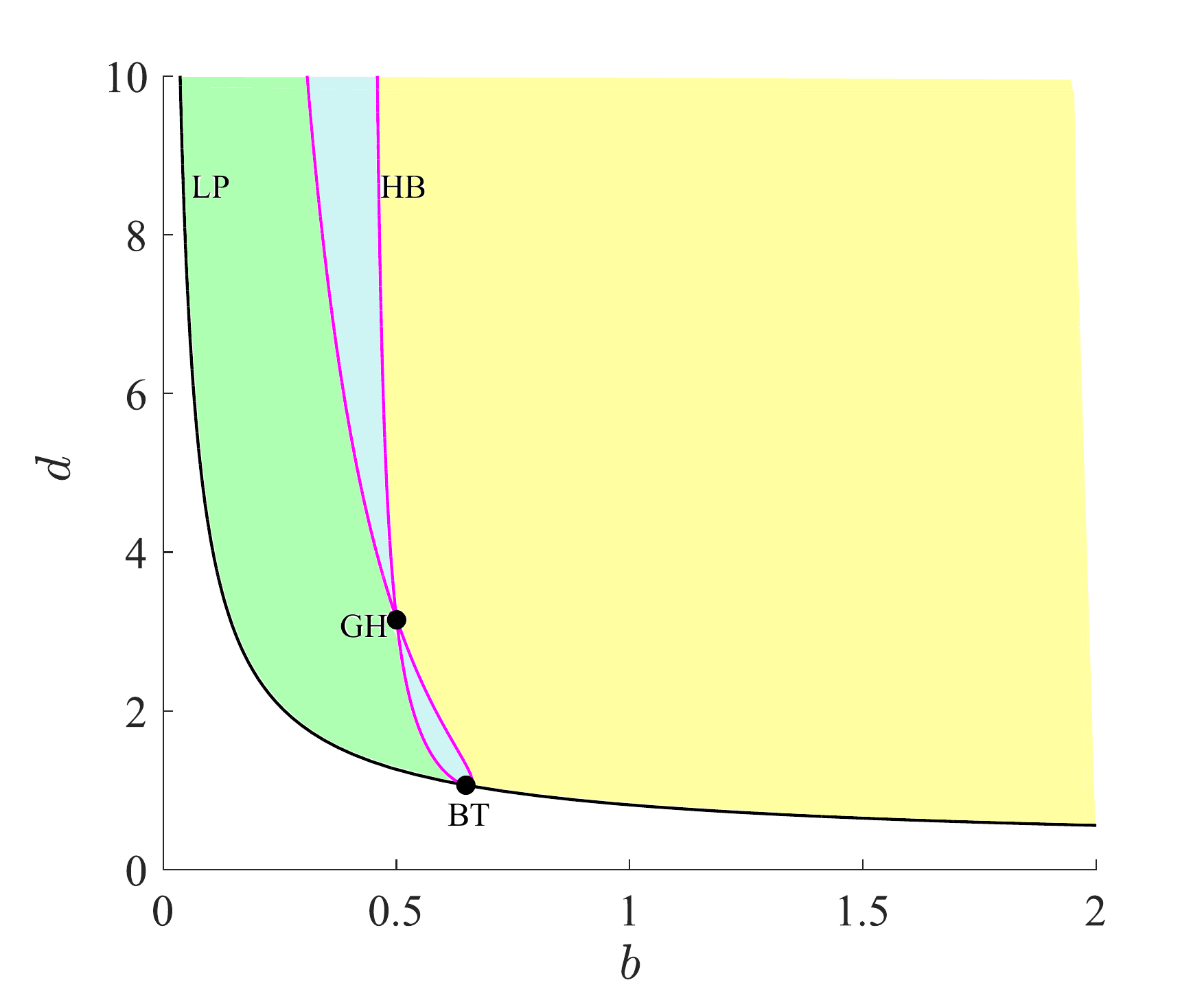}
  \end{subfigure}%
   \caption{Two parameter bifurcation diagram of \eqref{eq:5}--\eqref{eq:6} in the $(b,d)$-plane. other parameter values as in Table~\ref{table:t1}. The black and magenta curves are the loci of codimension-one bifurcations labelled as follows: SN: saddle-node bifurcation and HB: Hopf bifurcation}
       \label{fig:2par2}
\end{figure}

\section{Spatiotemporal dynamics} \label{s3}
In this section we explore the spatiotemporal patterns of \eqref{eq:m}--\eqref{eq:M} over one spatial domain. The system \eqref{eq:m}--\eqref{eq:M} is solved numerically using the method of lines technique. We use the second-order central finite difference scheme for the spatial derivative and a standard numerical scheme for the time derivative. In all simulations, we use zero-flux boundary conditions for $x\in[-L,L]$ and initial conditions
\begin{equation}
\label{eq:initial}
    m(x,0)=m^*+F(x) \hspace{1em} \text{and}\hspace{1em} M(x,0)=M^*+F(x),
\end{equation}
where $(m^*,M^*)$ is a homogeneous steady states of \eqref{eq:m}--\eqref{eq:M}. For all the initial conditions, the function $F(x)$ is a heterogenous perturbation from a homogeneous steady state. The perturbation is placed at the center of the domain. Unless otherwise stated, the diffusion coefficients $d_{1}$ and $d_{2}$ are given as $0.0001$ and $0.00001$, respectively.

\subsection{\textbf{Effects of varying $b$}}
 Now we consider the behaviour  of  model \eqref{eq:m}--\eqref{eq:M} varying $b$, other parameters are fixed as in Table \ref{table:t1}. In the initial condition we use
\begin{equation}
\label{eq:gauss}
    F(x)=\psi\,{\rm exp}\left( \frac{-x^2}{2\sigma^2} \right),
\end{equation}
with $\psi=0.3$ and $\sigma=0.1$.
 
 Figure~\ref{fig:pde_plot1} shows the spatiotemporal dynamics of $m$ for some values of $b$. For sufficiently low values of $b$, the perturbation initiates two counter-propagating pulses at the center of the domain. After a short time, the pulses transition to time-periodic oscillations as they move towards the boundary of the domain. For example in Fig.~\ref{fig:b=0.25-gauss} we have used $b=0.25$. As time progresses, the system goes back to the homogeneous steady state. With $b=0.5$ additional counter-propagating pulses are observed with complex spatiotemporal patterns at the back of the secondary pulses as they move across the domain. However, by increasing the value of $b$ further the pulses created at the centre by the initial perturbation transition to time-periodic oscillations which eventually spread to the entire domain. This results in synchronised oscillations of cells close to the boundaries and inhomogeneous patterns at the center, a typical example is shown in Fig,~\ref{fig:b=0.51-gauss}. Similar behaviour is observed in Figs.~\ref{fig:b=0.55-gauss} and \ref{fig:b=0.8-gauss} except for a  decrease in the amplitude oscillations as time progresses. For large values of $b$ the system returns quickly to a steady-state after a time-period oscillations for a short period. The results of numerical simulation for $b=2$ is shown in Fig.~\ref{fig:b=2.0-gauss}.
\begin{figure*}[htbp!]
\centering
\begin{subfigure}{.35\textwidth}
  \centering
  \caption{}
  \includegraphics[width = \textwidth]{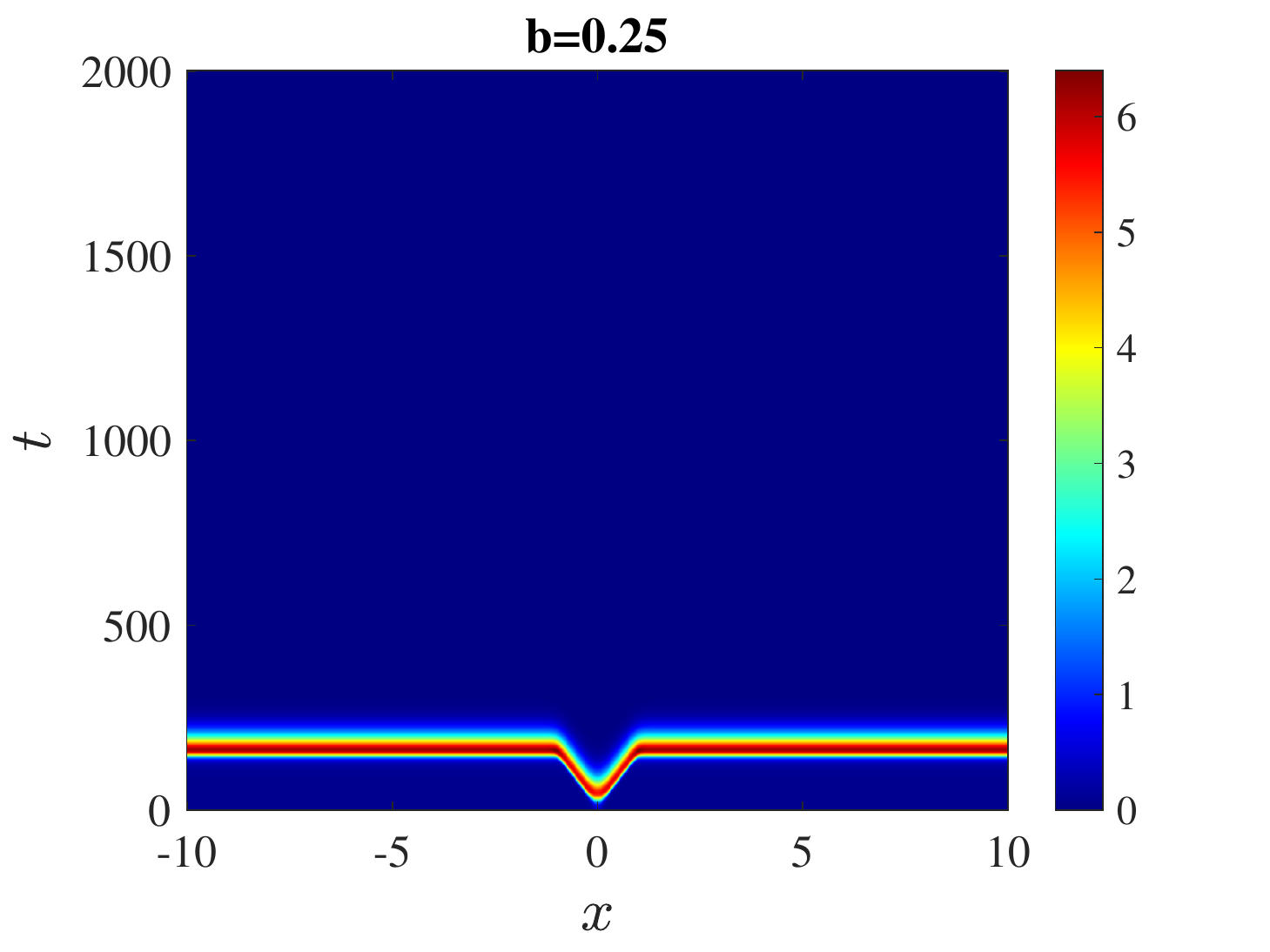}
  \label{fig:b=0.25-gauss}
\end{subfigure}%
\begin{subfigure}{.35\textwidth}
  \centering
  \caption{}
  \includegraphics[width = \textwidth]{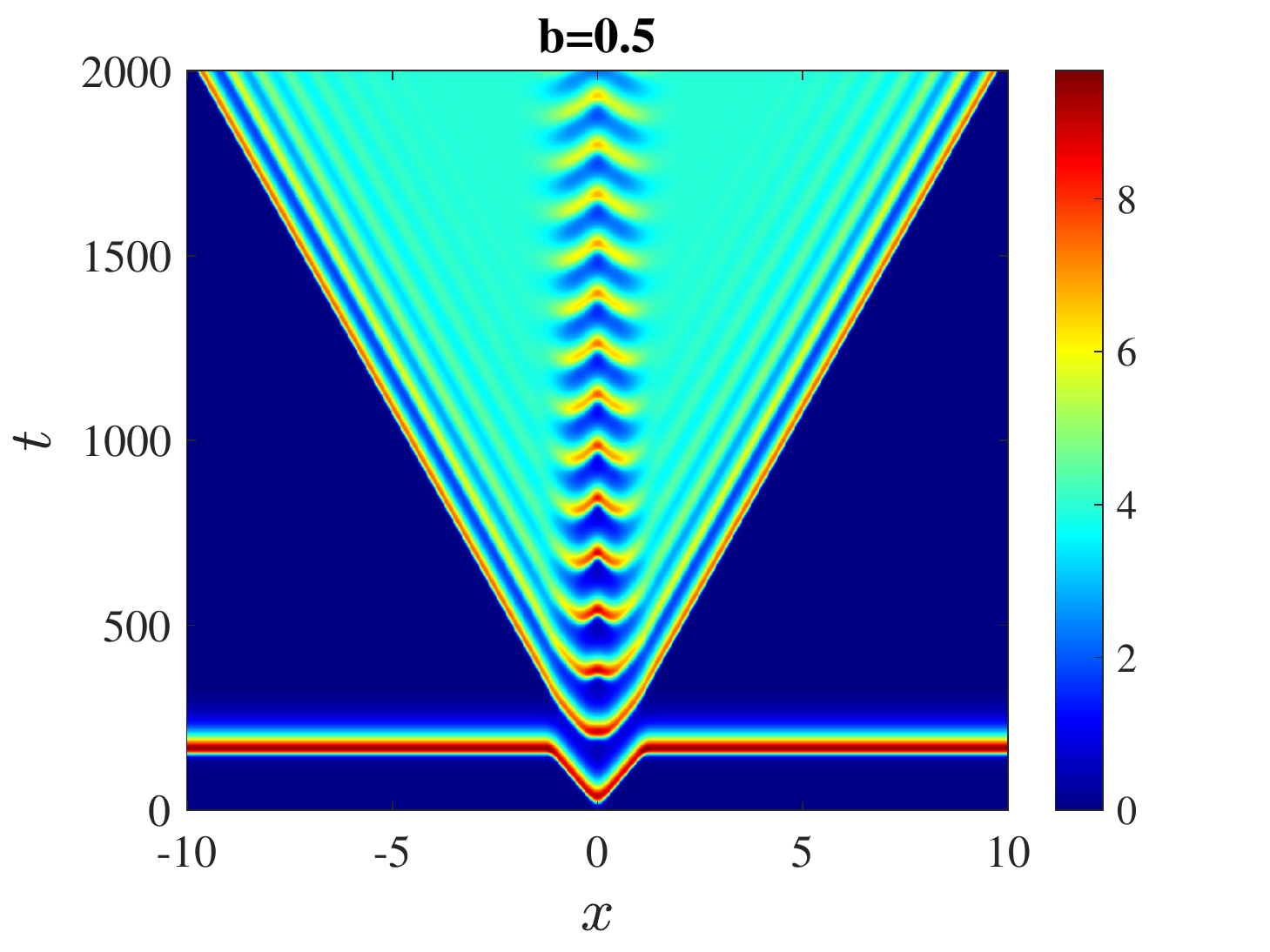}
  \label{fig:b=0.5-gauss}
\end{subfigure}%
\begin{subfigure}{.35\textwidth}
  \centering
  \caption{}
  \includegraphics[width = \textwidth]{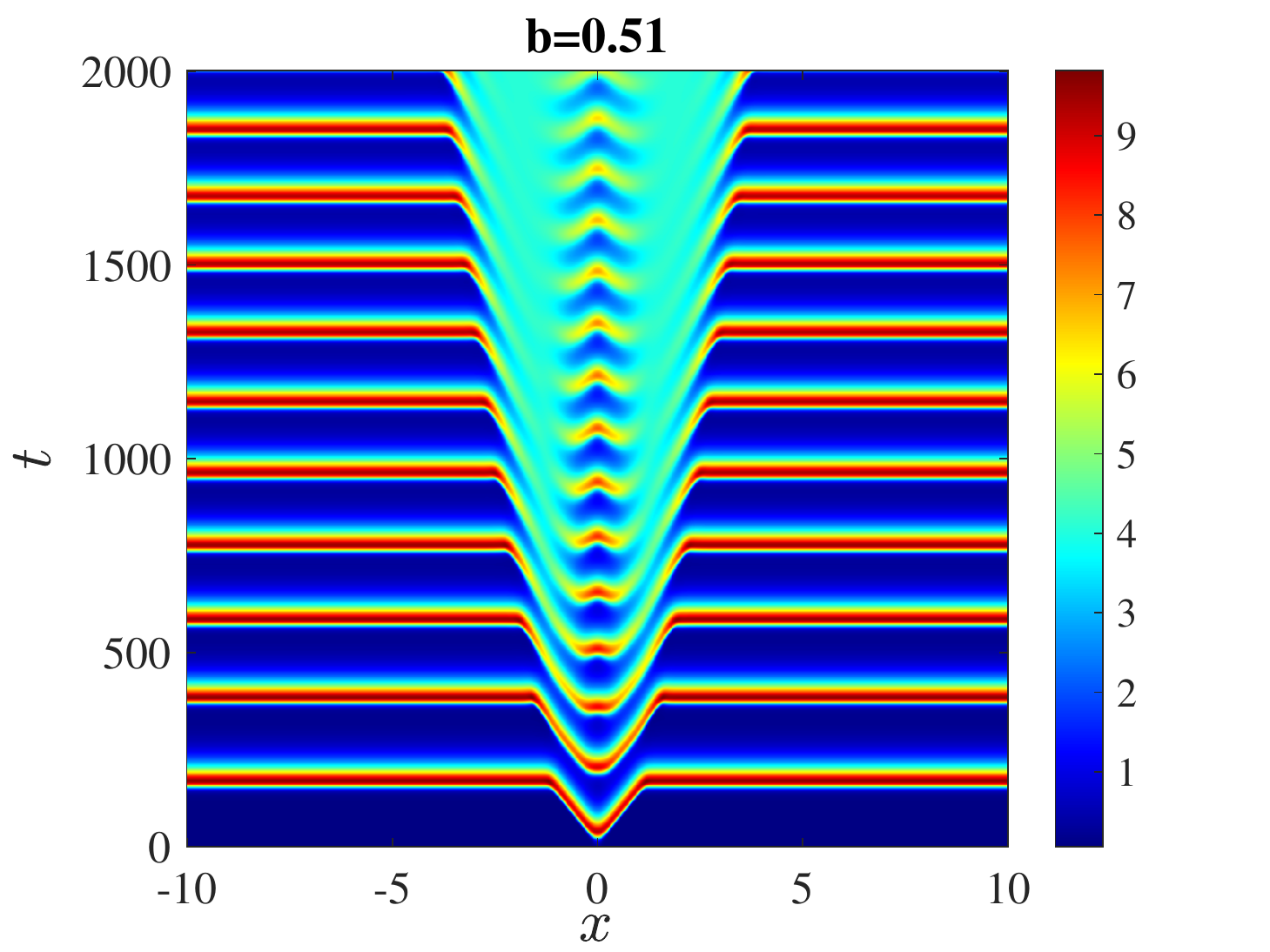}
   \label{fig:b=0.51-gauss}
\end{subfigure}\\%
\begin{subfigure}{.35\textwidth}
  \centering
  \caption{}
  \includegraphics[width = \textwidth]{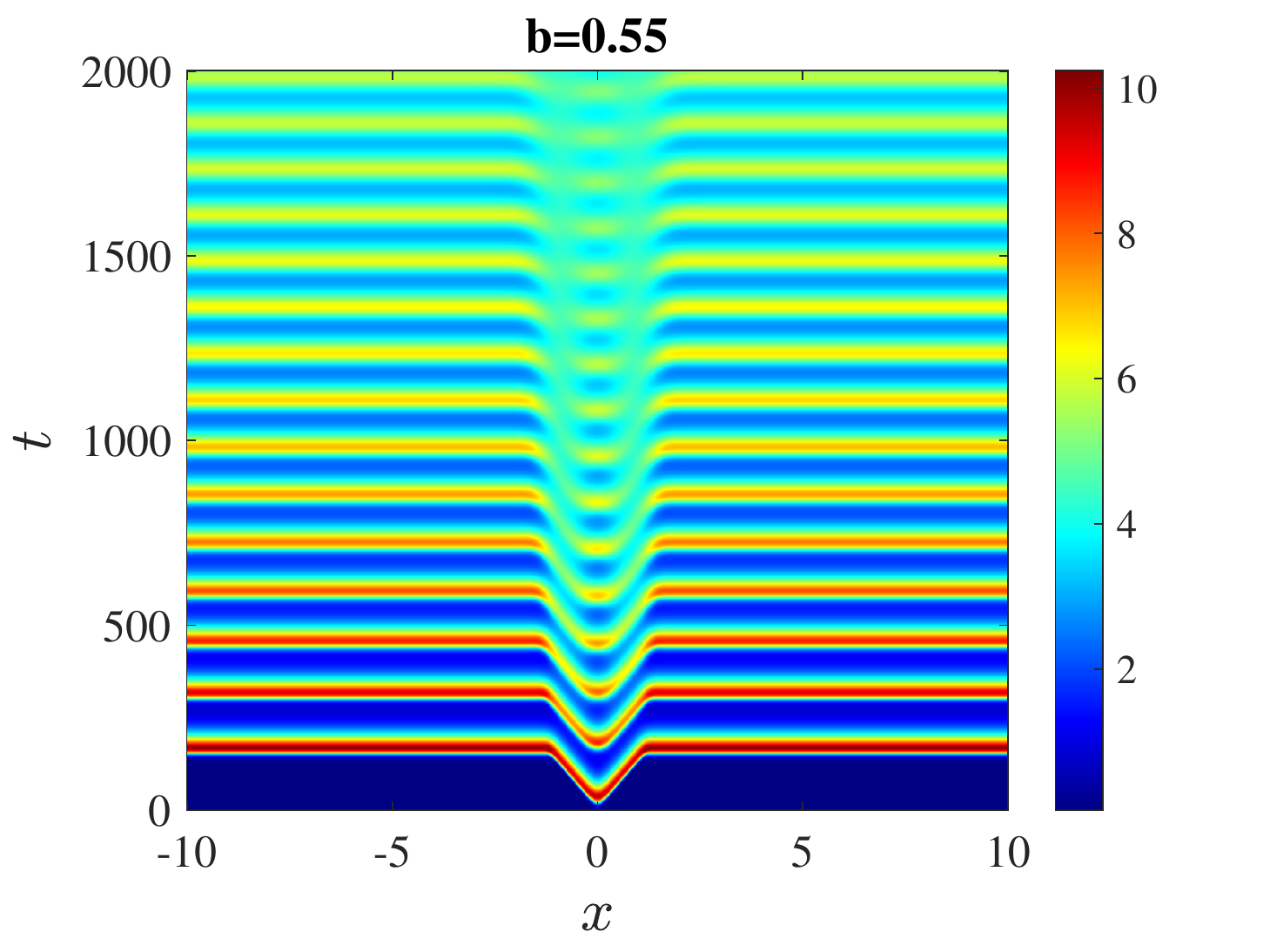}
   \label{fig:b=0.55-gauss}
\end{subfigure}%
\begin{subfigure}{.35\textwidth}
  \centering
  \caption{}
  \includegraphics[width = \textwidth]{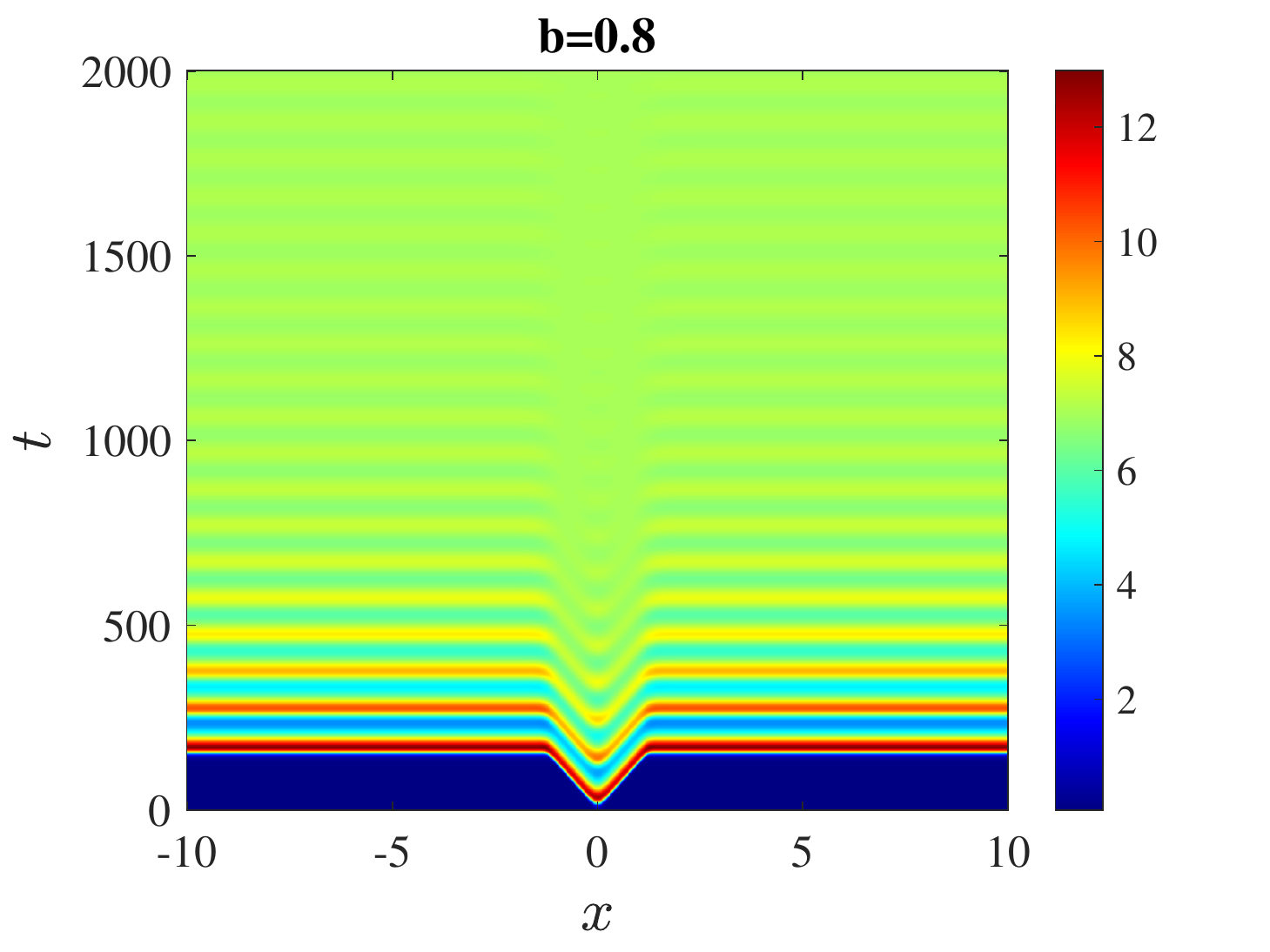}
   \label{fig:b=0.8-gauss}
\end{subfigure}%
\begin{subfigure}{.35\textwidth}
  \centering
  \caption{}
  \includegraphics[width = \textwidth]{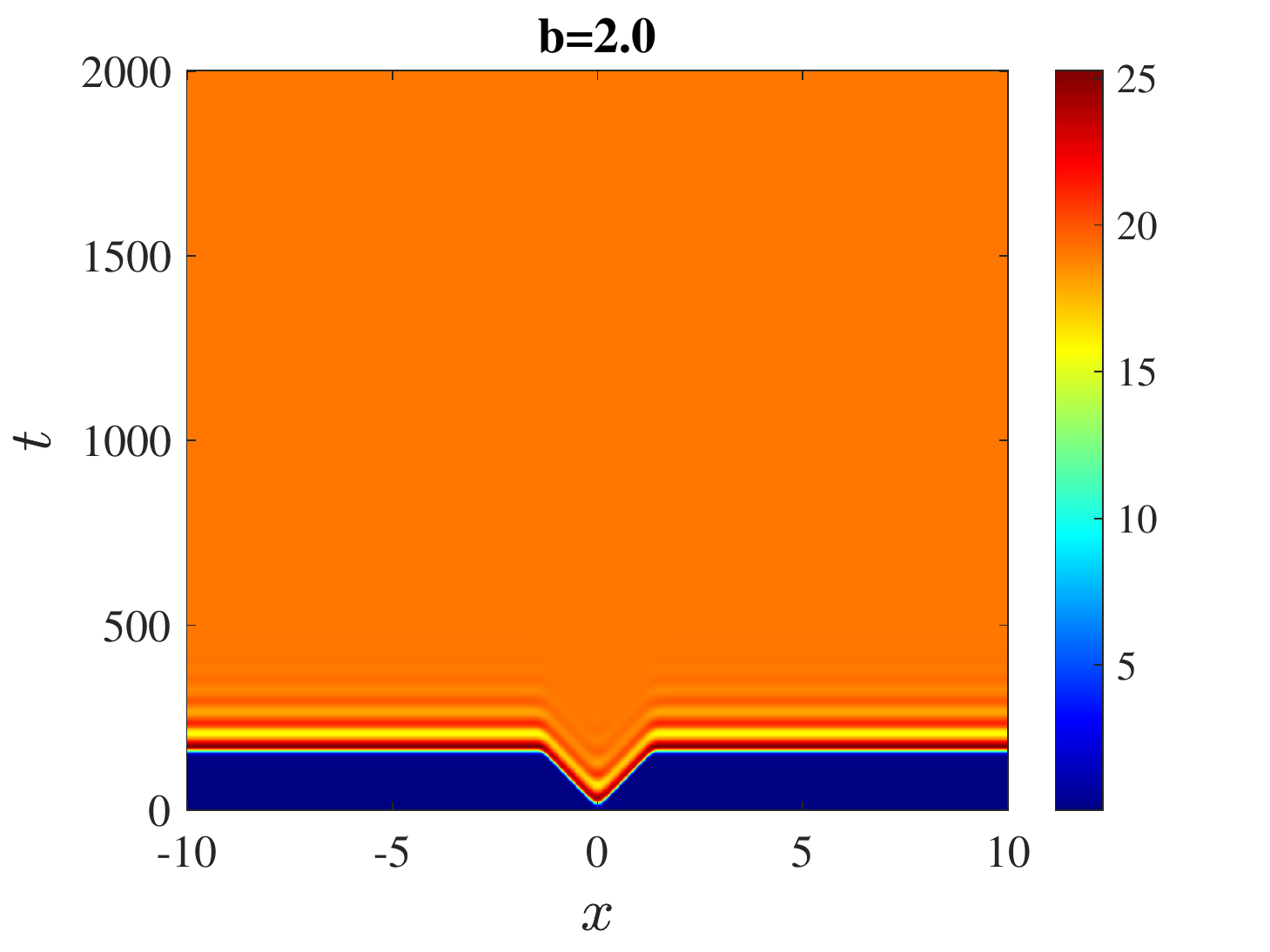}
   \label{fig:b=2.0-gauss}
\end{subfigure}%
\caption{Space-time plots of the $m$ for values of $b$. Specifically (a) $0.25$; (b) $0.5$; (c) $0.51$; (d) $0.55$; (e) $0.8$; and (f) $2.0$. The initial condition is $\eqref{eq:initial}$ with perturbation function $F(x)$ in $\eqref{eq:gauss}$,  $d_{1}=0.0001$, $d_{2}=0.00001$. Other parameters are fixed as in Table \ref{table:t1}}
\label{fig:pde_plot1}
\end{figure*}
\subsection{\textbf{Effect of changing diffusion coefficients $d_{1}$ and $d_{2}$}}
We now explore how selection of diffusion coefficients $d_{1}$ and $d_{2}$ influences the spatiotemporal dynamics of \eqref{eq:m}--\eqref{eq:M}. First we consider the case when $d_{1}=0.01$ and $d_{2}=0.0001$ with other parameters fixed as in Table~\ref{table:t1}. Fig.~\ref{fig:pde_plot2} shows the numerical results for three values of $b$. It is worth noting that the spatiotemporal patterns observed are similar to those in Fig.~\ref{fig:pde_plot1}a-c for the same values of $b$ except that the counter-propagating pulses initiated at the center of the domain spread quickly to the boundary. For example, in Fig.~\ref{fig:pde_plot1}b the primary and secondary propagating pulses reached the boundary at time t=2000s whereas in Fig.~\ref{fig:pde_plot2} (b), it takes lesser time for the propagating pulses to reach the boundary. For example, the primary pulses`   reached the boundary at time t=900s.
\begin{figure*}[htbp!]
\centering
\begin{subfigure}{.35\textwidth}
  \centering
  \includegraphics[width = \textwidth]{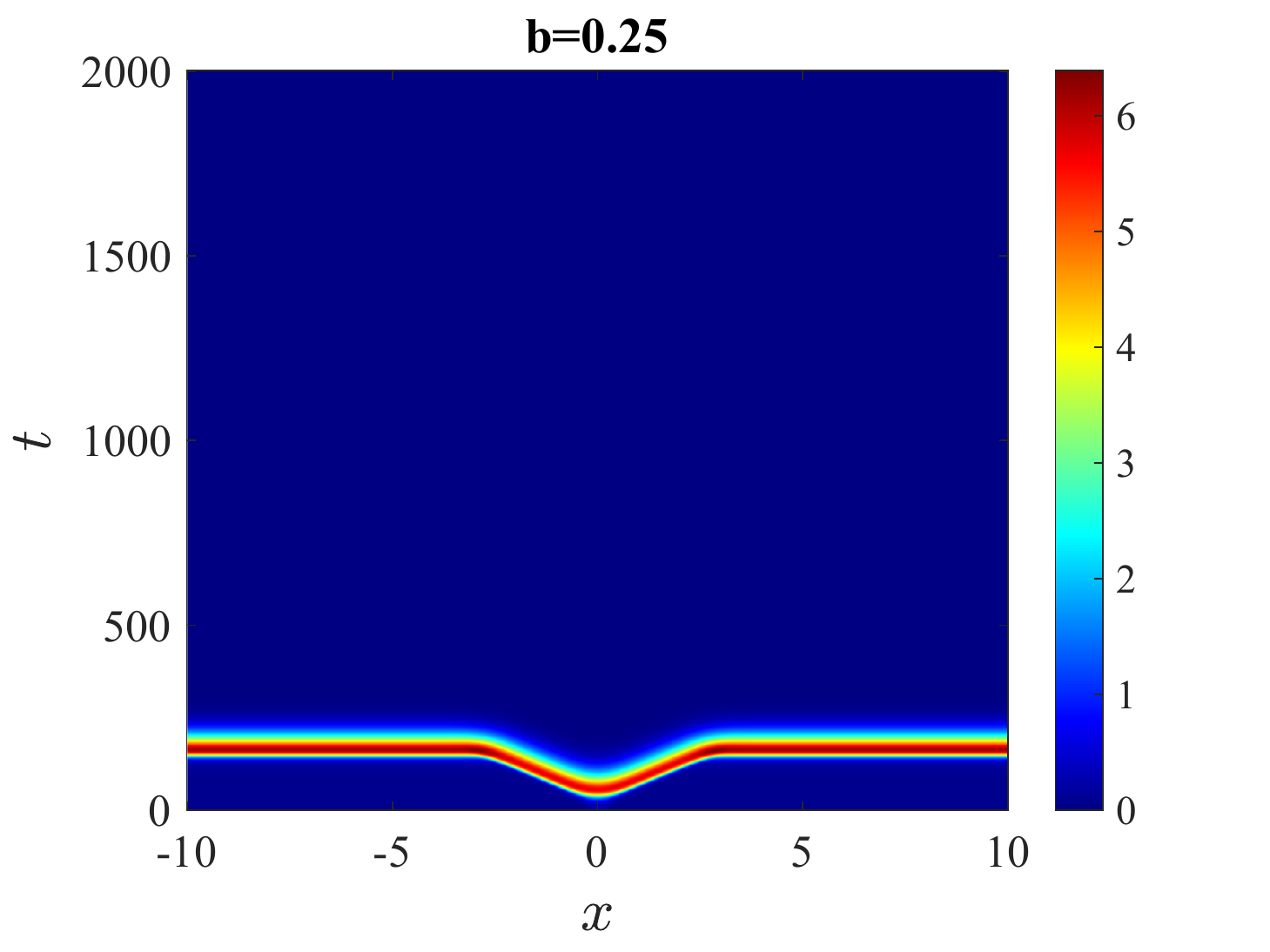}
  \caption{}
  \label{fig:b=0.25-gaussa}
\end{subfigure}%
\begin{subfigure}{.35\textwidth}
  \centering
  \includegraphics[width = \textwidth]{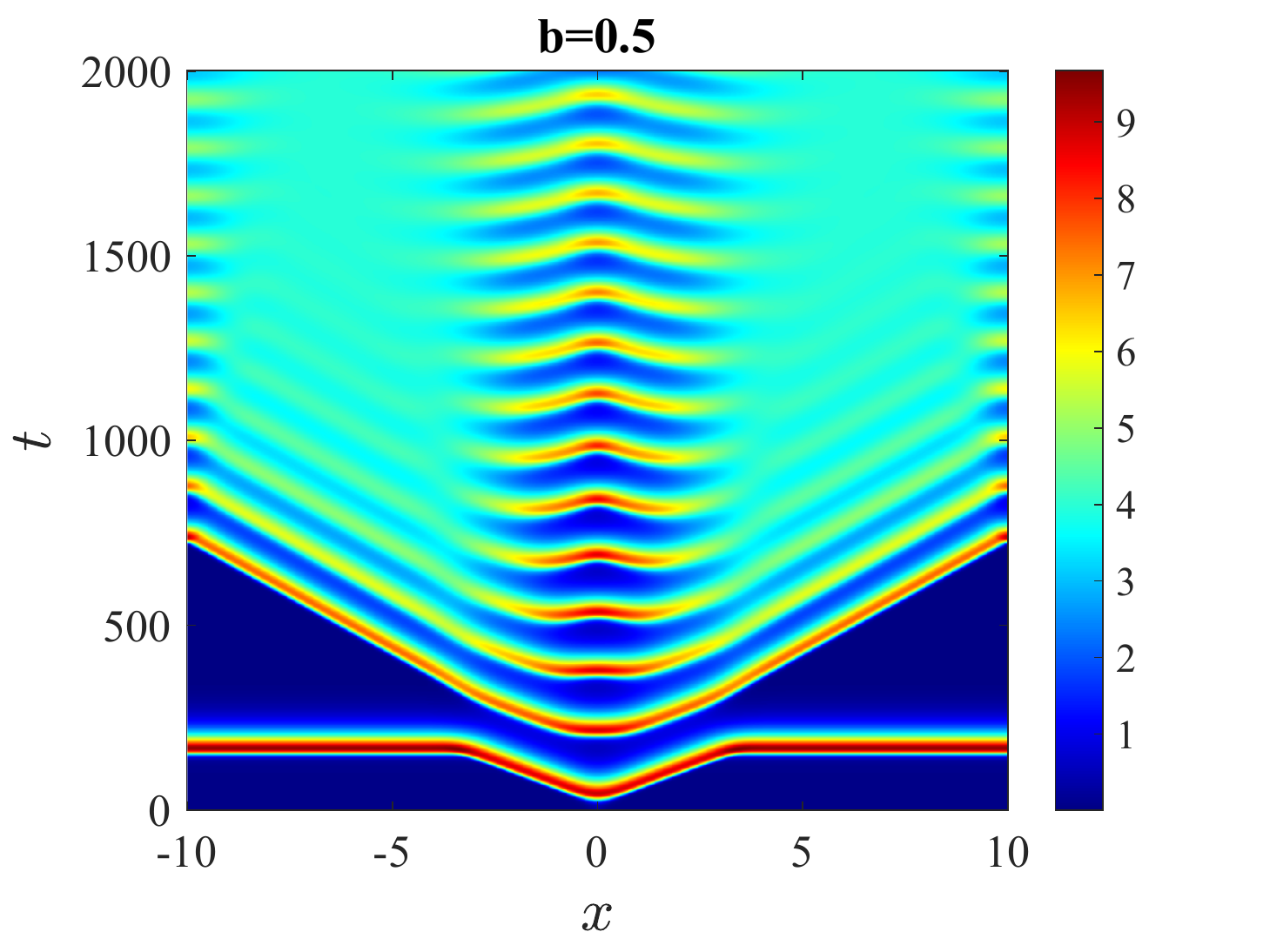}
  \caption{}
  \label{fig:b=0.5-gaussa}
\end{subfigure}%
\begin{subfigure}{.35\textwidth}
  \centering
  \includegraphics[width = \textwidth]{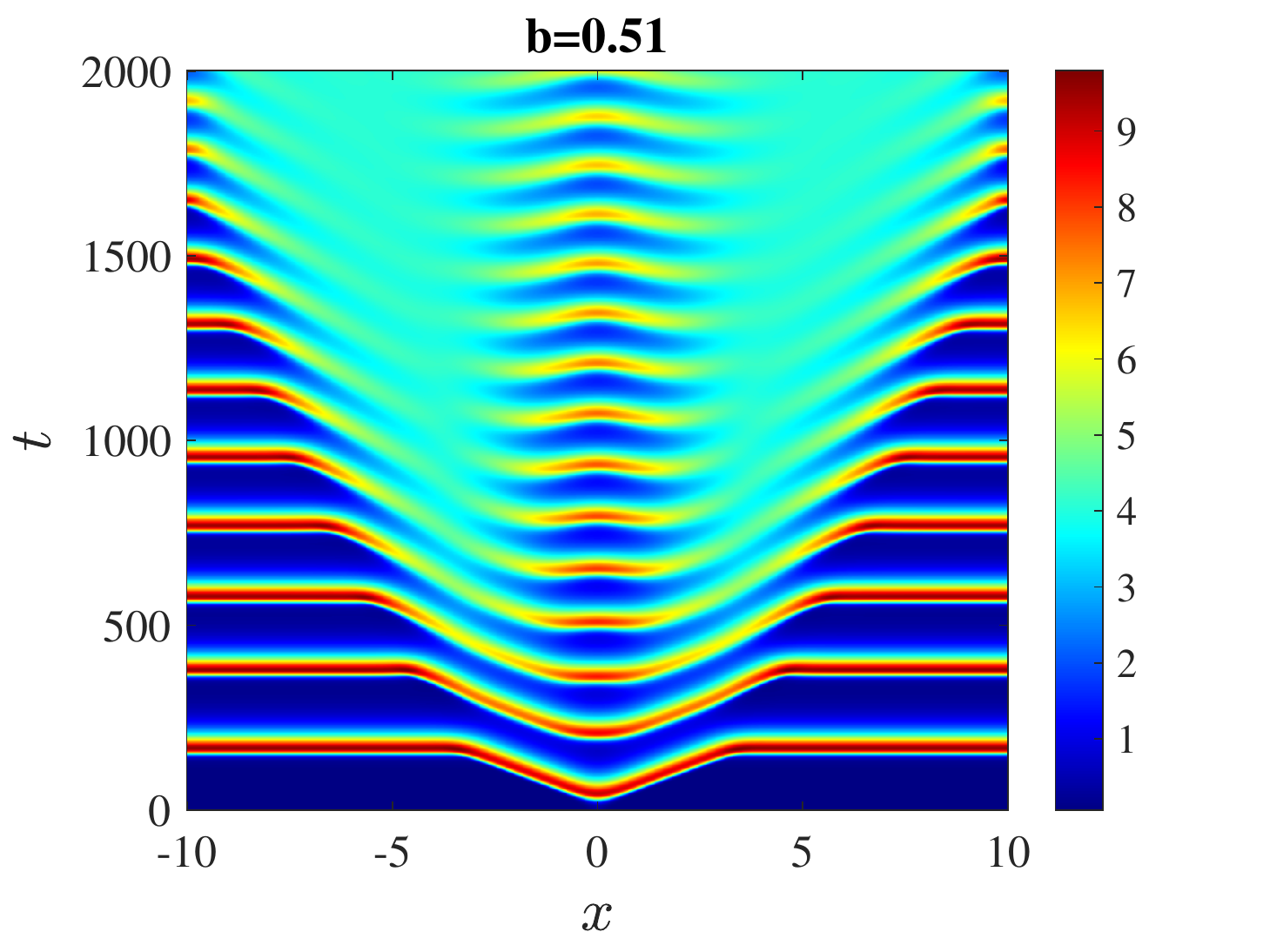}
  \caption{}
  \label{fig:b=0.51-gaussa}
\end{subfigure}\\%
\caption{Space-time plots of the $m$ for values of $b$. Specifically (a) $0.25$; (b) $0.5$; (c) $0.51$. The initial condition is $\eqref{eq:initial}$ with perturbation function $F(x)$ in $\eqref{eq:gauss}$, $d_{1}=0.01$, $d_{2}=0.0001$ and all other parameters are fixed as in Table \ref{table:t1}}
\label{fig:pde_plot2}
\end{figure*}


Lastly, setting $d_{1}=0.1$, $d_{2}=0.001$ produces the spatiotemporal patterns in Fig.~\ref{fig:pde_plot3}. All the solutions are oscillatory in time and homogeneous in space across the entire domain. For small $b$, a single pulse is observed for a short time before the system go back to the homogeneous steady state, see Fig.~\ref{fig:pde_plot3}a. Fig.~\ref{fig:pde_plot3} shows the system is synchronised with stable temporal oscillations. This corresponds to the stable limit cycle observed in the bifurcation diagram for the uncoupled cell in Fig.~\ref{fig:limitcycle}. For large $b$, temporal oscillations with amplitude decreasing as time progresses are observed. A typical example is shown in Fig.~\ref{fig:pde_plot3}c.
\begin{figure*}[htbp]
\centering
\begin{subfigure}{.35\textwidth}
  \centering
  \includegraphics[width = \textwidth]{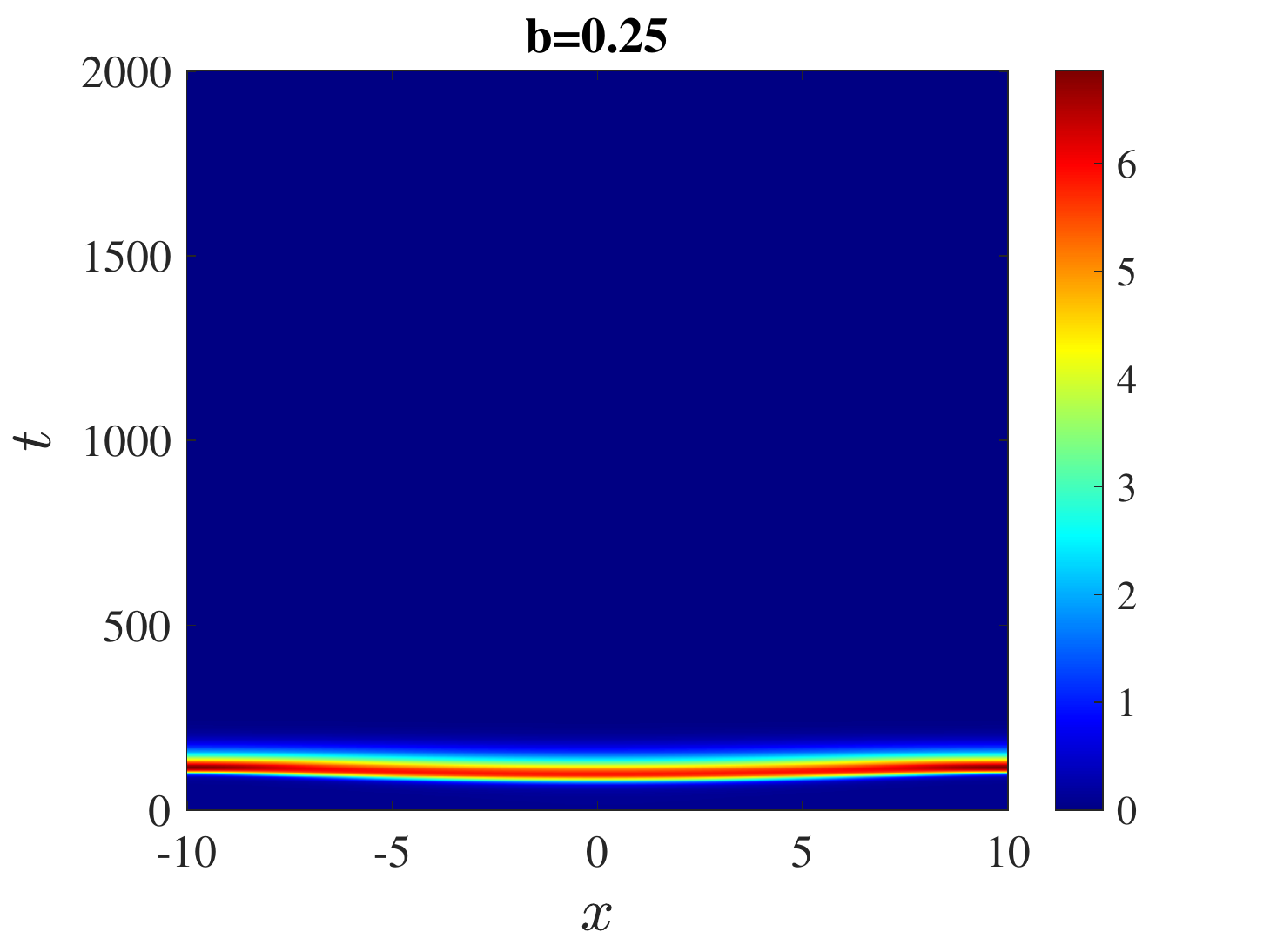}
  \caption{}
  \label{fig:b=0.25-gaussc}
\end{subfigure}%
\begin{subfigure}{.35\textwidth}
  \centering
  \includegraphics[width = \textwidth]{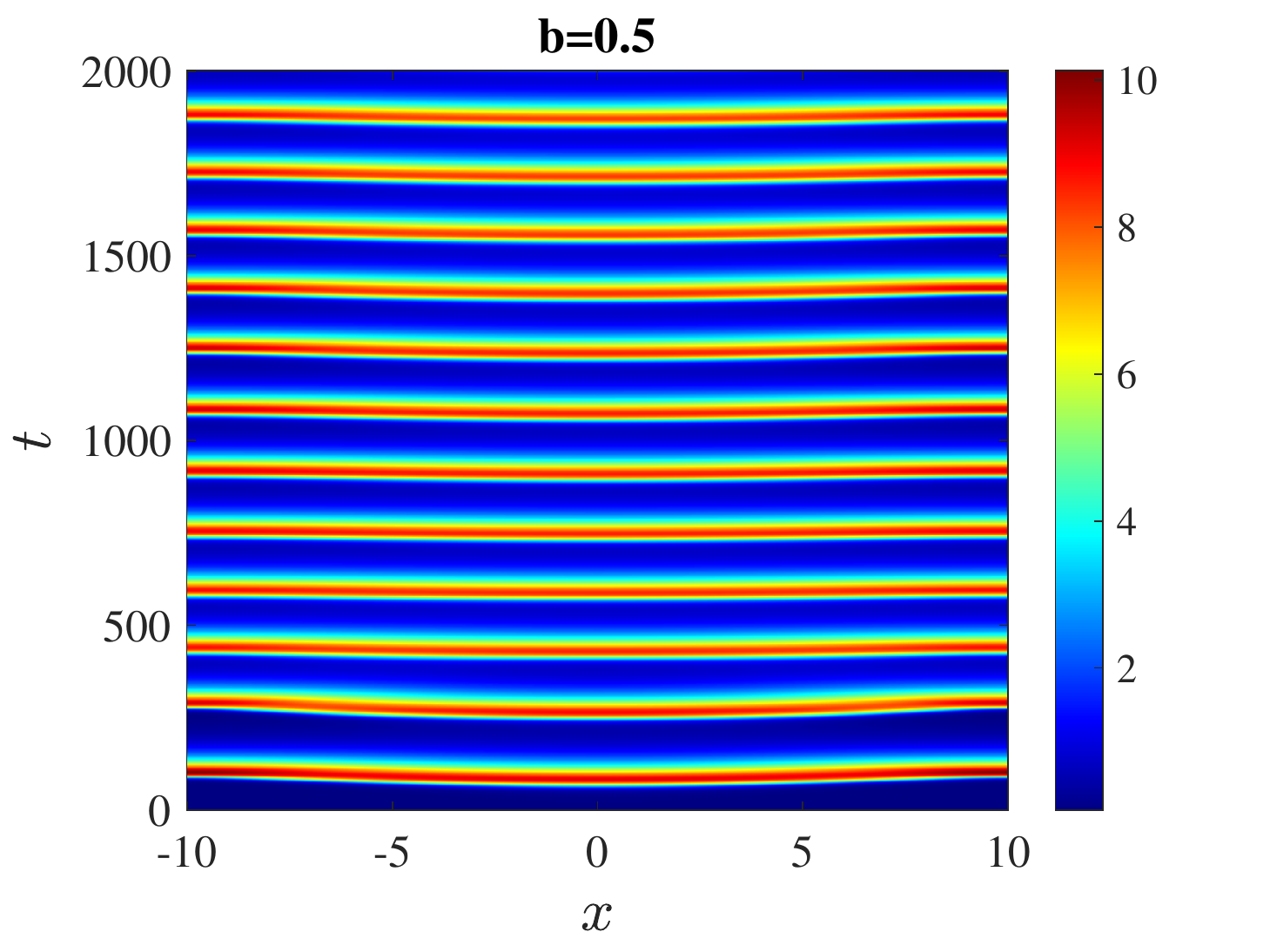}
  \caption{}
  \label{fig:b=0.5-gaussc}
\end{subfigure}%
\begin{subfigure}{.35\textwidth}
  \centering
  \includegraphics[width = \textwidth]{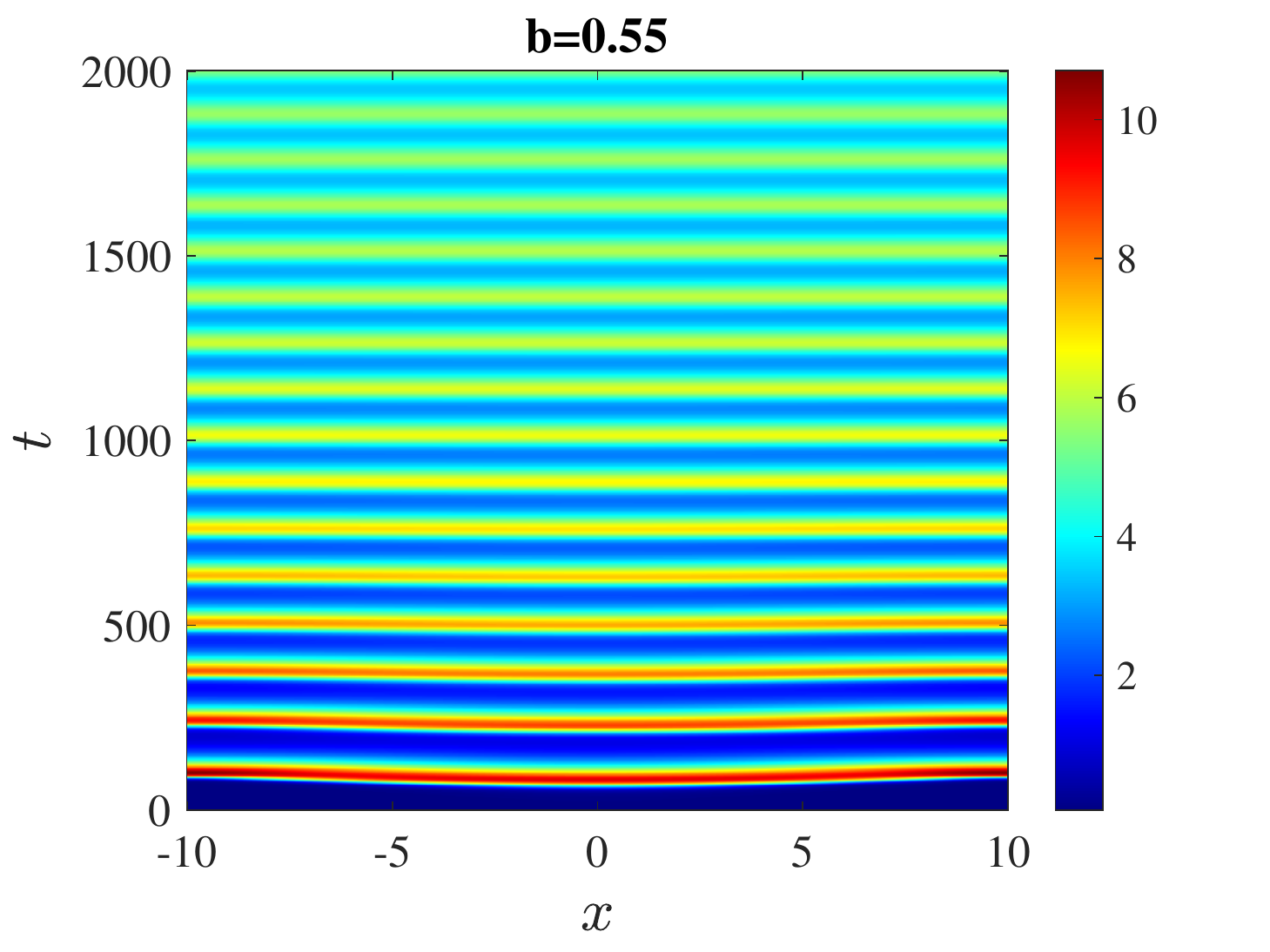}
  \caption{}
   \label{fig:b=0.55-gaussc}
\end{subfigure}\\%
\caption{Space-time plots of the $m$ for values of $b$. Specifically (a) $0.25$; (b) $0.5$; (c) $0.55$. The initial condition is $\eqref{eq:initial}$ with perturbation function $F(x)$ in $\eqref{eq:gauss}$, $d_{1}=0.1$, $d_{2}=0.001$ and all other parameters are fixed as in Table \ref{table:t1}.}
\label{fig:pde_plot3}
\end{figure*}


\section{Physical interpretation} \label{s4}

In this section the physical interpretation of Figs.~\ref{fig:MSNd3}-\ref{fig:pde_plot3} is being provided.We considered one- and two- parameter bifurcation analysis with the  ingestion rate of oxidized LDL by macrophages, b and the intake rate of oxidized LDL, d as bifurcating parameters. 

In Figs.~\ref{fig:MSNd3}-\ref{fig:2par2} we have observed that when $b\geq1$, the system depicts a stable nature. However for $b<1$ multiple cases have been investigated in Figs.~\ref{fig:MSNd3}-\ref{fig:2par2}. It suggests that the rate of ingestion of oxidized LDL by macrophages triggers larger plaque due to higher blood velocity ($b<1$). So the pertinent range of parameter b is $b>1$. Phase portraits for several values of the parameter b is provided in Fig.~\ref{fig:TT_Pplane}.  Fig.~\ref{fig:TT_Pplane}(a) is suggesting the stable nature phenomenon of the system for $b=1.5$.

We have observed unstable node in Fig.~\ref{fig:MSNd3}(a) (red colour zone), Fig.~\ref{fig:One-parameter bifurcation}(a) and Fig.~\ref{fig:One-parameter bifurcation}(b) for $0.03059 < b < 0.2314$. The region demonstrates that the ingestion rate of macrophages that are big enough to exhaust oxidized LDL in the plaque undergoes extension over a small period of time. After this period the plaque stabilizes again.

In Fig.~\ref{fig:2par1}(a) two stable equilibria and one unstable limit cycle is found in region II. The appearance of stable focus suggests that the plaque may evolve perpetually in accordance with the initial conditions. The plaque grows indefinitely in the presence of oscillation in the concentration of plaque ingredients while keeping the initial conditions of monocytes and macrophages within the region of the focus. A limited plaque growth is observed due to a short-term oscillations in the concentration of plaque ingredients, while the initial conditions for monocytes and macrophages are assumed to be outside the zone of attraction.

Fig.~\ref{fig:2par2} depicts that as the intake rate of oxidized LDL $d$ increases, the volume of the plaque surges.

Figs.~\ref{fig:pde_plot1}--\ref{fig:pde_plot3} for the PDE model shows the existence of travelling waves in the system. It validates the fact that the chronic inflammation in atherosclerosis is a result of propagation of reaction-diffusion waves.  Fig.~\ref{fig:pde_plot1} represents the counter - propagating pulses at different values of the parameter, b. It suggests that when the domain width is sufficiently large the wave propagation occurs near the surfaces due to the presence of excess amount of monocytes. Fig.~\ref{fig:pde_plot2} shows the existence of travelling waves for different values of b while keeping the diffusion parameters as $d_1= 0.01$ and $d_2= 0.0001$.  The diffusion parameter values are increased from   Fig.~\ref{fig:pde_plot1} to Fig.~\ref{fig:pde_plot2} and hence the monocyte concentration at the boundaries are significantly perturbed. In Fig. \ref{fig:pde_plot2} the counter propagating pulses are reaching the boundaries way faster than in Fig. \ref{fig:pde_plot1} which is a clear indication of the presence of excess monocytes at the boundaries. In Fig.~\ref{fig:pde_plot3} the values of the diffusion parameters are further increased to $d_1= 0.1$ and $d_2= 0.001$. It validates the temporal oscillations observed with respect to b in Fig.~\ref{fig:limitcycle}.

\section{Conclusion} \label{s5}

The advancement of atherosclerosis is being investigated here by deploying several tools from dynamical system theory. The role of macrophage phagocytosis in the plaque forming process is explored in terms of a temporal and a spatiotemporal model. A detailed linear stability analysis has been performed for the temporal model. Two kinds of bifurcation analysis are being considered, namely, one parametric and two parametric. The ingestion rate of oxidized LDL to macrophages $b$ and the intake rate of oxidized LDL $d$ are considered as the bifurcating parameters. We have shown the spatiotemporal model is dissipative in nature. It depicts the existence of bistability in the system. Numerical investigation on the spatiotemporal model validates the fact that $b>1$ is the ideal parametric range.

The results presented here show that there is a period of time when the plaque components oscillate at the onset of atherosclerosis. After this period, the plaque evolution process accelerates significantly. The bifurcation analysis is helpful in recognizing parameter regimes which are clinically significant. The results can be communicated with the clinical researchers for further investigations and  generating suitable therapeutic strategies in controlling the disease dynamics.

\bibliographystyle{spbasic}      

\bibliography{referbio.bib}

\end{document}